\documentclass[12pt]{article}
\usepackage{amsthm,amsmath,amssymb}
\usepackage{fullpage}
\usepackage{float}
\usepackage{tikz}

\newtheorem{thm}{Theorem}[section]
\newtheorem{lem}[thm]{Lemma}
\newtheorem{prop}[thm]{Proposition}

\newtheorem{claim}{Claim}

\newcommand{\HH}{\mathcal H}
\newcommand{\FF}{\mathcal F}
\newcommand{\FA}{\mathcal A}
\newcommand{\BB}{\mathcal B}
\newcommand{\NH}{N_{\HH}}
\newcommand{\R}{\mathbb R}

\begin{document}

\title{The maximum spectral radius of outerplanar and planar $k$-uniform hypergraphs}
\author{Pei Liu\thanks{Department of Mathematics, Sungkyunkwan University, Suwon, 16419, Republic of Korea. liupei2023@g.skku.edu. Research supported by the China Scholarship Council.} \and Suil O\thanks{Department of Applied Mathematics and Statistics, The State University of New York, Korea, Incheon, 21985, suil.o@sunykorea.ac.kr. Corresponding author. Research supported by the National Research Foundation of Korea (NRF) grant funded by the Korea government(MSIT) No. RS-2025-23523950.}}
\date{\today}

\maketitle

\begin{abstract}
\noindent
For $k\ge3$, a $k$-angulation is a simple $2$-connected outerplane graph whose interior faces are bounded by $k$-cycles, and a closed $k$-angulation is a simple $2$-connected plane graph all of whose faces are bounded by $k$-cycles. Their face hypergraphs have the same vertices as the graphs; the hyperedges are the vertex sets of the interior faces, respectively all faces. For $k=3$, the extremal problems over these hypergraphs agree with those for the outerplanar and planar $3$-uniform hypergraphs of Ellingham, Lu and Wang, who determined the outerplanar maximizer for large $n$ and conjectured the planar one. In this paper, for each fixed $k$, we determine the hypergraphs of maximum spectral radius in both classes for all sufficiently large $n$ satisfying the corresponding face-count divisibility condition. In the outerplanar case, the unique extremal hypergraph is the face hypergraph of the fan, in which one vertex lies on every interior face. In the planar case with $k=3$, it is the face hypergraph of $K_2+P_{n-2}$, which confirms the conjecture of Ellingham, Lu and Wang. For $k\ge4$, the extremal hypergraphs are the face hypergraphs of balanced theta graphs, in which two vertices lie on every face. For $k=4$, the result holds for every $n\ge5$, with extremal hypergraph $\HH(K_{2,n-2})$. For $k\ge6$ and an even number of faces, there are exactly $\lfloor(k-2)/2\rfloor$ extremal hypergraphs up to isomorphism.\\


\end{abstract}

\section{Introduction}\label{sec:intro}

The study of the spectral radius of planar and outerplanar graphs goes back to Schwenk and Wilson~\cite{SW}. For planar graphs, Hong~\cite{H0} gave the first bound of the correct order. Boots and Royle~\cite{BR} and independently Cao and Vince~\cite{CV} conjectured that $K_2+P_{n-2}$ is extremal. Guiduli and Hayes~\cite{GH} announced a proof for large $n$, Tait and Tobin~\cite{TT} proved it for large $n$, and Liu, Ning and Wang~\cite{LNW} settled it for every $n\ge9$. For outerplanar graphs, Cvetkovi\'c and Rowlinson~\cite{CR} conjectured that the fan $K_1+P_{n-1}$ is extremal; Tait and Tobin~\cite{TT} proved this for large $n$, and Lin and Ning~\cite{LN} for every $n\ne6$. Ellingham, Lu and Wang~\cite{ELW} carried the question to $3$-uniform hypergraphs. They showed that, for large $n$, the outerplanar maximizer is the hypergraph whose shadow is $K_1+P_{n-1}$, and conjectured the planar analogue with $K_2+P_{n-2}$. In this paper, we prove this conjecture and determine the maximizers for every fixed uniformity in the face-hypergraph classes defined below.

For an integer $k\ge3$, a \emph{$k$-angulation} is a simple $2$-connected outerplane graph all of whose interior faces are bounded by cycles of length $k$. A \emph{closed $k$-angulation} is a simple $2$-connected plane graph all of whose faces, including the outer face, are bounded by cycles of length $k$. The \emph{face hypergraph} $\HH(G)$ has vertex set $V(G)$ and hyperedges given by the vertex sets of the interior faces of a $k$-angulation, or of all faces of a closed $k$-angulation. We take these hypergraphs to be simple, and this costs nothing: if two faces had the same vertex set, their boundary cycles would form a simple $2$-connected plane graph $J$ on $k$ vertices having both as faces, and with $m=e(J)$, Euler's formula and the face lengths would give $2m\ge2k+3(m-k)$, so $m\le k$ and $J$ is a single $k$-cycle. The two faces would then be the two sides of this cycle, which is impossible for interior faces of a $k$-angulation and forces $G=C_k$ for a closed $k$-angulation. Hence, apart from $G=C_k$, distinct faces have distinct vertex sets, and summing over faces counts each hyperedge once. By Euler's formula, a $k$-angulation on $n$ vertices has $f=(n-2)/(k-2)$ interior faces, whereas a closed $k$-angulation has $2(n-2)/(k-2)$ faces. We call $n$ \emph{admissible} for the outerplanar problem when $(k-2)\mid(n-2)$, and for the closed problem when $(k-2)\mid2(n-2)$. A $k$-uniform hypergraph is a \emph{$k$-angulation hypergraph} if it is the face hypergraph of a $k$-angulation. For $k\ge4$, we call it \emph{planar} if it is the face hypergraph of a closed $k$-angulation.

Eigenvalues of tensors were introduced by Qi~\cite{Q} and Lim~\cite{Li}, and the adjacency tensor of a uniform hypergraph is due to Cooper and Dutle~\cite{CD}. The \emph{spectral radius} of a $k$-uniform hypergraph $\HH$ on $n$ vertices is
\[
\lambda(\HH)=\max\Bigl\{P_{\HH}(x):\ x\in\R^{n}_{\ge0},\ \sum_v x_v^{k}=1\Bigr\},\qquad
P_{\HH}(x)=k\sum_{e\in E(\HH)}\ \prod_{v\in e}x_v .
\]
When $\HH$ is connected, the maximum is attained at a unique positive vector under this normalization. This \emph{Perron vector} satisfies
\begin{equation}\label{eq:eigen}
\lambda(\HH)\,x_v^{\,k-1}=\sum_{e\ni v}\ \prod_{u\in e\setminus\{v\}}x_u\qquad(v\in V(\HH));
\end{equation}
see~\cite{CPZ,FGH}.

Our definitions use Zykov's face-based approach~\cite{Z}, in which a plane graph carries the hyperedges as faces. For $k=3$, we use the notion of Ellingham, Lu and Wang~\cite{ELW}: $\HH$ is outerplanar, respectively planar, if its \emph{shadow} $\partial\HH$, the graph of pairs lying in a common hyperedge, has an outerplanar, respectively plane, embedding in which each hyperedge is the vertex set of an interior face, respectively a triangular face. For $k\ge4$, this shadow formulation admits no hypergraph with an edge: a hyperedge spans $K_k$ in the shadow, $K_5$ is nonplanar, and every plane embedding of $K_4$ has only triangular faces. We therefore use a plane carrier whose faces are the hyperedges. For $k\ge4$, our planar class consists specifically of the face hypergraphs of closed $k$-angulations; hypergraphs carried by only some faces of a plane graph with other face lengths are not included. Within a fixed closed $k$-angulation $G$, taking all faces loses no extremal hypergraph, since $\HH\subsetneq\HH(G)$ implies $\lambda(\HH)<\lambda(\HH(G))$ when $\HH(G)$ is connected. For $k=3$, the shadow can be augmented to a maximal outerplanar graph~\cite{ELW} or a plane triangulation (Lemma~\ref{lem:augment}) while preserving the prescribed faces. Thus the extremal problems over face hypergraphs agree with those of~\cite{ELW}.

We now describe the extremal hypergraphs. The \emph{fan $k$-angulation} $\FF_{k,f}$ is the $k$-angulation with $f$ interior faces in which one vertex, the \emph{hub}, lies on every interior face; for $k=3$, it is $K_1+P_{n-1}$. For $N\ge2$, let $D_N$ be the face hypergraph of the plane triangulation $K_2+P_N$. For $k\ge4$, $t\ge3$ and $p_1,\dots,p_t\ge1$ with $p_i+p_{i+1}=k-2$ for all $i$ (indices modulo $t$), let $\Theta(p_1,\dots,p_t)$ be the plane graph in which two \emph{branch vertices} are joined by internally disjoint paths with $p_1,\dots,p_t$ internal vertices, in this cyclic order. Each of its $t$ faces is a $k$-cycle through both branch vertices, and it has $n=2+\frac t2(k-2)$ vertices. We call it a \emph{balanced theta graph}; see Figure~\ref{fig:theta}. The condition gives $p_{i+2}=p_i$, so the $p_i$ alternate between two positive integers summing to $k-2$. For odd $t$, they all equal $(k-2)/2$. Such a graph exists for every $n$ admissible for the closed problem with $t=2(n-2)/(k-2)\ge3$, since $kt$ is even.

\begin{figure}[H]
\centering
\begin{tikzpicture}[scale=0.95,
br/.style={circle,fill=black,inner sep=2.4pt},
iv/.style={circle,draw=black,fill=white,inner sep=1.5pt}]
\begin{scope}
\node[br] (a) at (0,2.1) {}; \node[above] at (a) {$h_1$};
\node[br] (b) at (0,-2.1) {}; \node[below] at (b) {$h_2$};
\foreach \s/\x in {1/-2.4, 2/-0.8, 3/0.8, 4/2.4}
{\node[iv] (p\s) at (\x,0.7) {}; \node[iv] (q\s) at (\x,-0.7) {};
\draw (a)--(p\s)--(q\s)--(b);}
\node at (0,-3.0) {(a) $p_i=2$ for all $i$};
\end{scope}
\begin{scope}[xshift=7.4cm]
\node[br] (A) at (0,2.1) {}; \node[above] at (A) {$h_1$};
\node[br] (B) at (0,-2.1) {}; \node[below] at (B) {$h_2$};
\foreach \s/\x in {1/-2.4, 2/0.8}
{\node[iv] (r\s) at (\x,0) {}; \draw (A)--(r\s)--(B);}
\foreach \s/\x in {1/-0.8, 2/2.4}
{\node[iv] (s\s) at (\x,1.05) {}; \node[iv] (t\s) at (\x,0) {}; \node[iv] (w\s) at (\x,-1.05) {};
\draw (A)--(s\s)--(t\s)--(w\s)--(B);}
\node at (0,-3.0) {(b) $p_i=1,3,1,3$};
\end{scope}
\end{tikzpicture}
\caption{Two nonisomorphic balanced theta graphs with $k=6$, $t=4$ and $n=10$; by Proposition~\ref{prop:values} their face hypergraphs have the same spectral radius.}\label{fig:theta}
\end{figure}

\begin{thm}\label{thm:main}
Let $k\ge3$ be fixed.
\begin{itemize}
\item[$(a)$] There is $f_1=f_1(k)$ such that, for every integer $f\ge f_1$, $\HH(\FF_{k,f})$ is, up to isomorphism, the unique $k$-angulation hypergraph on $n=(k-2)f+2$ vertices of maximum spectral radius. This maximum is $(4f)^{1/k}(1+o(1))$.
\item[$(b)$] For $k=3$ and all sufficiently large $n$, an $n$-vertex planar $3$-uniform hypergraph has maximum spectral radius if and only if it is isomorphic to $D_{n-2}$.
\item[$(c)$] For $k=4$ and every $n\ge5$, every planar $4$-uniform hypergraph $\HH$ on $n$ vertices satisfies $\lambda(\HH)\le\sqrt{2(n-2)}$, with equality if and only if $\HH\cong\HH(K_{2,n-2})$. For $k\ge5$ and all sufficiently large $n$ admissible for the closed problem, a planar $k$-uniform hypergraph on $n$ vertices has maximum spectral radius if and only if it is the face hypergraph of a balanced theta graph. The maximum is $2^{1-2/k}t^{2/k}$, where $t=2(n-2)/(k-2)$. In particular, for $k\ge6$ and even $t$ there are exactly $\lfloor(k-2)/2\rfloor$ pairwise nonisomorphic extremal hypergraphs, whereas for $k\in\{4,5\}$ or odd $t$ there is exactly one.
\end{itemize}
\end{thm}

The order of the maximum depends on both the embedding class and the uniformity. In the outerplanar problem, one vertex can lie on every interior face, and the maximum has order $n^{1/k}$. In a closed $k$-angulation with $k\ge4$, two vertices can lie on every face, whereas three distinct vertices lie on at most two common faces (Lemma~\ref{lem:three}). We prove that, in an extremal hypergraph, exactly two Perron coordinates stay bounded away from zero as $n\to\infty$, and the maximum has order $n^{2/k}$. For $k=3$, two vertices share at most two faces, balanced theta graphs are unavailable, and the order is $n^{1/3}$. So for $k\ge4$ the planar extremal structure differs from that of $K_2+P_{n-2}$: the two dominant vertices share every face rather than two, and for $k\ge6$ the extremal hypergraph need not be unique.

Each part has two steps. A stability argument first locates almost all faces at one or two dominant vertices. A direct comparison with the proposed extremal hypergraph then excludes the remaining faces. The comparison is where the hypergraph problems depart from the graph case. The proof in~\cite{ELW} bootstraps the eigenequation by counting coefficients in the second neighborhood of the hub, which works for $k=3$ because every non-hub vertex of the fan is adjacent to the hub. For $k\ge4$, each face of the fan contains, besides the hub, two spokes and $k-3$ interior vertices; the interior vertices are not adjacent to the hub and carry a different weight at the optimum, so this count no longer controls the eigenequation. We replace it by assignments of vertices to faces with bounded multiplicity: an exact partition obtained from the dual tree in the outerplanar problem, and a flow argument in the planar problem. The planar problem for $k=3$ needs a different comparison, since no two vertices share more than two faces: the faces avoiding the two hubs are controlled through the discs they cut out of the triangulation, and a path-shortening argument shows that their contribution is smaller than the loss they cause at the two hubs. Two restrictions used throughout follow from the nonplanarity of $K_{3,3}$: three vertices share at most two faces, and only boundedly many light vertices can serve several pairs of heavy vertices.

Section~\ref{sec:def} collects the tools. Subsections~\ref{sec:outer}, \ref{sec:pk4} and~\ref{sec:p3} prove Theorem~\ref{thm:main}$(a)$, $(c)$ and $(b)$, respectively, and Section~\ref{sec:remark} contains concluding remarks. For undefined graph-theoretic terms, see West~\cite{W}; for spectral graph theory, see Brouwer and Haemers~\cite{BH} or Godsil and Royle~\cite{GR}.

\section{Definitions and Tools}\label{sec:def}

We use \emph{face} to mean interior face when $G$ is a $k$-angulation, and identify a face with its vertex set when no confusion arises. For a vertex $v$, let $d(v)$ be the number of faces containing $v$, and let $\NH(v)$ be the set of vertices other than $v$ lying on a face with $v$. For ordinary graph neighborhoods and degrees, we use $N_G(v)$ and $\deg_G(v)$; we abbreviate $N_G(v)$ to $N(v)$ when the graph is clear. When $k=3$, some sets below are empty; empty products are interpreted as $1$.

Let $G$ be a $k$-angulation with $f$ interior faces. Its edges on the outer cycle are the \emph{boundary edges}; the remaining $f-1$ edges are the \emph{chords}. Each chord lies on two interior faces and each boundary edge on one. The \emph{dual tree} $T$ has the interior faces as nodes, with two nodes adjacent when their faces share a chord. Root $T$ at a face $F_0$. For $F\ne F_0$, let the \emph{parent chord} $c(F)$ be the chord shared with its parent, and let $O(F)=F\setminus c(F)$ be the set of $k-2$ vertices \emph{owned} by $F$; put $O(F_0)=F_0$.

\begin{lem}\label{lem:local}
Let $G$ be a $k$-angulation and $v\in V(G)$.
\begin{itemize}
\item[$(i)$] The faces containing $v$ can be listed as $F_1,\dots,F_{d}$, where $d=d(v)$, so that consecutive faces share a chord through $v$; they induce a path in $T$.
\item[$(ii)$] If $u\ne v$ lie on two distinct faces $F,F'$, then $uv$ is a chord and $F,F'$ are its two faces. In particular, two vertices lie on at most two common faces, and two nonconsecutive faces at $v$ meet only in $v$.
\item[$(iii)$] If $u\ne v$, then $|\NH(u)\cap\NH(v)|\le2k$.
\item[$(iv)$] For every choice of the root, the sets $O(F)$ partition $V(G)$.
\end{itemize}
\end{lem}

\begin{proof}
$(i)$ The edges at $v$ occur in the rotation as $vw_1,\dots,vw_{d+1}$, with $vw_1,vw_{d+1}$ the boundary edges. The sector between $vw_i$ and $vw_{i+1}$ is a face $F_i$, so $F_i$ and $F_{i+1}$ share the chord $vw_{i+1}$.

$(ii)$ By $(i)$, if two faces contain a vertex $w$, every chord corresponding to an edge of the path between them in $T$ contains $w$. Thus every edge of the path from $F$ to $F'$ corresponds to the chord $uv$. Distinct edges of $T$ correspond to distinct chords, so this path is a single edge.

$(iii)$ Let $P_u,P_v$ be the paths of faces at $u,v$, and let $w\in F\cap F'$ with $F\in P_u$, $F'\in P_v$. Every chord on the path from $F$ to $F'$ contains $w$. If $P_u\cap P_v=\emptyset$, this path leaves $P_u$ through the fixed chord corresponding to the first edge of the unique path joining $P_u$ to $P_v$. There are at most two possible vertices $w$ on that chord. Otherwise the path passes through a face of $P_u\cap P_v$, so $w$ lies on one of the at most two common faces.

$(iv)$ Let $F^*$ be the face of $P_v$ nearest the root. If $F^*=F_0$, then $v\in O(F_0)$. Otherwise the parent of $F^*$ does not contain $v$, so $v\in O(F^*)$. For every other face $F$ of $P_v$, the path to the root starts along $P_v$ toward $F^*$. By $(i)$, its parent chord contains $v$, so $v\notin O(F)$.
\end{proof}

The next three lemmas concern plane graphs. The first replaces Lemma~\ref{lem:local}$(ii)$, which fails for closed $k$-angulations when $k\ge4$.

\begin{lem}\label{lem:three}
In a closed $k$-angulation, three distinct vertices lie on at most two common faces.
\end{lem}

\begin{proof}
If $u,v,w$ lay on three faces, place a point in each face and join it within that face to $u,v,w$. This would embed $K_{3,3}$ in the sphere.
\end{proof}

\begin{lem}\label{lem:amb}
Let $G$ be a plane graph in which every face has at most $k$ vertices on its boundary, and let $h_1,h_2,h_3$ be distinct vertices. Let $S$ consist of the vertices $v\notin\{h_1,h_2,h_3\}$ such that, for each $i$, some face contains $v$ and $h_i$ but not all of $h_1,h_2,h_3$. Then $|S|\le6k-10$.
\end{lem}

\begin{proof}
The vertex--face incidence graph $I$ of $G$ is planar. For $v\in S$, choose faces $F_i(v)\ni v,h_i$ with $\{h_1,h_2,h_3\}\not\subseteq F_i(v)$. The \emph{core} of $v$, consisting of $v$ and the face-vertices $F_i(v)$, is connected in $I$ and avoids $h_1,h_2,h_3$. Three pairwise disjoint cores would give a $K_{3,3}$ minor after contraction, since $F_i(v)h_i\in E(I)$. Each $F_i(v)$ contains $v$, $h_i$, and hence at most $k-2$ other vertices of $S$. Thus at most $3(k-2)$ other vertices of $S$ have cores meeting the core of $v$. Choose $v_1$, discard it and every vertex whose core meets its core, and repeat once if any vertex remains. No vertex can remain after the second step, so $|S|\le2(3(k-2)+1)=6k-10$.
\end{proof}

\begin{lem}\label{lem:augment}
Let $\HH$ be a planar $3$-uniform hypergraph on $n\ge4$ vertices. A witnessing plane embedding of $\partial\HH$ can be augmented to a plane triangulation $T$ without adding vertices or disturbing a facial triangle that carries a hyperedge. Consequently, every planar $3$-uniform hypergraph of maximum spectral radius on $n$ vertices is the face hypergraph of a plane triangulation.
\end{lem}

\begin{proof}
Call the faces carrying hyperedges \emph{protected}; we add edges only through unprotected faces. Join two components on the boundary of a common face through that face. Such a face is unprotected, since a protected face is bounded by a single triangle. At a cut vertex $v$, choose neighbors $u,w$ consecutive in the rotation at $v$ and lying in distinct components of $G-v$. Then $uw\notin E(G)$, and the face between $vu$ and $vw$ is unprotected, since a protected face there would contain $uw$. Add $uw$ through this face. Continuing gives a $2$-connected plane graph, so every face is bounded by a cycle. A facial cycle $C$ of length at least four has two nonconsecutive vertices that are not adjacent: otherwise $V(C)$ would span $K_5$, or a plane $K_4$ with $C$ as a face. Add such a diagonal. The process ends with a triangulation $T$ and $\HH\subseteq\HH(T)$. Since $\HH(T)$ is connected, extremality forces equality.
\end{proof}

To estimate the weight of a family of faces, we split each face into $k-2$ vertices and a remaining pair. A \emph{$(c,c')$-assignment} for a family $\FF$ of $k$-faces of a plane graph assigns to each $F\in\FF$ a set $\phi(F)\subseteq F$ of size $k-2$ such that $F\setminus\phi(F)$ is an edge of the boundary cycle of $F$, each vertex lies in at most $c$ assigned sets, and each edge occurs as the remaining pair for at most $c'$ faces.

\begin{lem}\label{lem:assign}
Let $\FF$ be a family of faces.
\begin{itemize}
\item[$(i)$] If $G$ is a $k$-angulation and $\FF$ omits at least one face, then $\FF$ has a $(1,1)$-assignment.
\item[$(ii)$] If $G$ is a closed $3$-angulation, then $\FF$ has a $(2,2)$-assignment.
\item[$(iii)$] If $G$ is a closed $k$-angulation and $H_0$ is a set of at most one vertex, there are sets $\phi(F)\subseteq F\setminus H_0$ of size $k-2$, for $F\in\FF$, such that each vertex lies in at most six of them.
\end{itemize}
\end{lem}

\begin{proof}
$(i)$ Root $T$ at a face not in $\FF$ and take $\phi(F)=O(F)$. These sets are disjoint by Lemma~\ref{lem:local}$(iv)$, and each chord is the parent chord of exactly one face.

$(ii)$ Here $|\phi(F)|=1$, and the remaining pair is an edge lying on at most two faces. Join each face to two copies of each of its vertices. For a nonempty subfamily $\FF'\subseteq\FF$ with vertex set $U$, counting face--edge incidences and using planarity gives $3|\FF'|\le2e(G[U])\le2(3|U|-6)$. In particular, $|\FF'|\le2|U|$, so Hall's theorem gives the assignment.

$(iii)$ Use a network with capacity $k-2$ from a source to each face, capacity one from a face to each of its vertices outside $H_0$, and capacity six from each vertex to a sink. By the integral max-flow min-cut theorem, it suffices to show
\[
\sum_{v\notin H_0}\min\{6,d_{\mathcal D}(v)\}\ge(k-2)|\mathcal D|
\]
for every nonempty $\mathcal D\subseteq\FF$, where $d_{\mathcal D}(v)$ counts the faces of $\mathcal D$ at $v$. Put $X=\{v\notin H_0:d_{\mathcal D}(v)>6\}$. If $X=\emptyset$, the sum is at least $k|\mathcal D|-|\mathcal D|$. Otherwise $|\mathcal D|\ge7$. The incidence graph $J$ between $\mathcal D$ and $X\cup H_0$ is simple, planar and bipartite, so $e(J)\le2(|\mathcal D|+|X|+|H_0|)-4$. Hence
\[
\sum_{v\notin H_0}\min\{6,d_{\mathcal D}(v)\}=k|\mathcal D|-e(J)+6|X|\ge(k-2)|\mathcal D|+4|X|+2.
\]
\end{proof}

\begin{lem}\label{lem:general}
Let $\FF$ be a family of $m$ faces of a plane graph $G$ admitting a $(c,c')$-assignment. Let $x\in\R^{V(G)}_{\ge0}$, let $V(\FF)$ be the set of vertices on faces of $\FF$, and put $W=\sum_{v\in V(\FF)}x_v^{k}$.
\begin{itemize}
\item[$(i)$] We have
\[
\sum_{F\in\FF}\prod_{v\in F}x_v\le\Bigl(\frac{c}{k-2}\Bigr)^{(k-2)/k}\Bigl(\frac{c'^{2}m}{2}\Bigr)^{1/k}W.
\]
\item[$(ii)$] If $G$ is $D$-degenerate and $x_v\le\varepsilon$ on $V(\FF)$, then
\[
\sum_{F\in\FF}\prod_{v\in F}x_v\le\Bigl(\frac{c}{k-2}\Bigr)^{(k-2)/k}(c'D)^{2/k}\varepsilon\,|V(\FF)|^{1/k}W^{(k-1)/k}.
\]
\end{itemize}
In particular, if $G$ is a $k$-angulation on $n$ vertices, then
\begin{equation}\label{eq:facesout}
\begin{aligned}
\sum_{F\in\FF}\prod_{v\in F}x_v&\le(k-2)^{-(k-2)/k}(m/2)^{1/k}W+\rho W/k,\\
\sum_{F\in\FF}\prod_{v\in F}x_v&\le\frac{2^{2/k}\varepsilon\,n^{1/k}W^{(k-1)/k}}{(k-2)^{(k-2)/k}}+\rho\,\varepsilon^{k},
\end{aligned}
\end{equation}
where the second bound assumes $x_v\le\varepsilon$ on $V(\FF)$, and $\rho=0$ if $\FF$ omits a face and $\rho=1$ otherwise. If $G$ is a closed $3$-angulation, then
\begin{equation}\label{eq:facesp3}
\sum_{F\in\FF}\prod_{v\in F}x_v\le(4m)^{1/3}W,\qquad\sum_{F\in\FF}\prod_{v\in F}x_v\le200^{1/3}\varepsilon\,|V(\FF)|^{1/3}W^{2/3},
\end{equation}
with the same coordinate restriction for the second bound.
\end{lem}

\begin{proof}
Write $F\setminus\phi(F)=a_Fb_F$ and $Y_F=\prod_{v\in\phi(F)}x_v$. H\"older's inequality with exponents $k/(k-2)$ and $k/2$ gives
\[
\sum_Fx_{a_F}x_{b_F}Y_F\le\Bigl(\sum_FY_F^{\frac k{k-2}}\Bigr)^{\frac{k-2}{k}}\Bigl(\sum_F(x_{a_F}x_{b_F})^{\frac k2}\Bigr)^{\frac2k}.
\]
By the arithmetic--geometric mean inequality, $Y_F^{k/(k-2)}\le(k-2)^{-1}\sum_{v\in\phi(F)}x_v^{k}$, so the first sum is at most $cW/(k-2)$. Put $y_v=x_v^{k/2}$ and let $J$ be the simple graph of the pairs $a_Fb_F$, a subgraph of $G$ with $e(J)\le m$. Then $\sum_Fy_{a_F}y_{b_F}\le(c'/2)\,y^{\top}A(J)y$, where $A(J)$ is the adjacency matrix of $J$. Write $\lambda_1(J)$ for its largest eigenvalue, the adjacency spectral radius. For $(i)$, use $y^{\top}A(J)y\le\lambda_1(J)W$ and $\lambda_1(J)^2\le2e(J)\le2m$. For $(ii)$, orient $J$ with outdegrees at most $D$. Since $y_v\le\varepsilon^{k/2}$, the Cauchy--Schwarz inequality gives
\[
\sum_{uv\in E(J)}y_uy_v\le D\varepsilon^{k/2}\sum_vy_v\le D\varepsilon^{k/2}\sqrt{|V(\FF)|W}.
\]

For~\eqref{eq:facesout}, if $\FF$ contains every face, remove one face $F_0$. By the arithmetic--geometric mean inequality, its weight is at most $k^{-1}\sum_{v\in F_0}x_v^k\le W/k$, or at most $\varepsilon^k$ under the coordinate restriction. The remaining family has a $(1,1)$-assignment by Lemma~\ref{lem:assign}$(i)$. Outerplanar graphs are $2$-degenerate, and $|V(\FF)|\le n$. For~\eqref{eq:facesp3}, apply Lemma~\ref{lem:assign}$(ii)$ with $D=5$.
\end{proof}

For the planar problem with $k\ge4$, we also need an estimate for many faces containing a fixed pair.

\begin{lem}\label{lem:pair}
Let $e$ be a pair of vertices, and let $F_1,\dots,F_t$ be faces containing $e$. Put $A_i=F_i\setminus e$ and $M_e=\sum_i\sum_{v\in A_i}x_v^{k}$. Then
\[
\sum_i\prod_{v\in A_i}x_v\le t^{2/k}\Bigl(\frac{M_e}{k-2}\Bigr)^{(k-2)/k}.
\]
\end{lem}

\begin{proof}
Each $A_i$ has $k-2$ elements, so $\bigl(\prod_{v\in A_i}x_v\bigr)^{k/(k-2)}\le(k-2)^{-1}\sum_{v\in A_i}x_v^{k}$. Apply H\"older's inequality with exponents $k/2$ and $k/(k-2)$.
\end{proof}

In the fan $\FF_{k,d}$ with hub $h$ and faces $F_1,\dots,F_d$ in order, let $u_0,\dots,u_d$ be the spokes: $F_i\cap F_{i+1}=\{h,u_i\}$, and $u_0,u_d$ are the outer neighbors of $h$. Let $I_i=F_i\setminus\{h,u_{i-1},u_i\}$ be the $k-3$ interior vertices of $F_i$. For $y\ge0$ on the non-hub vertices and $\tau\ge0$, define the \emph{fan functional} by
\[
\Phi_{d,\tau}(y)=\sum_{i=1}^{d}\ \prod_{v\in F_i\setminus\{h\}}y_v+\tau(y_{u_0}+y_{u_d}),\qquad
\Theta_{d,\tau}(W)=\max_{\substack{y\ge0\\ \sum_vy_v^{k}=W}}\Phi_{d,\tau}(y).
\]
Write $\Theta_d=\Theta_{d,0}$ and $\theta_d=\Theta_d(1)$, so $\Theta_d(W)=\theta_dW^{(k-1)/k}$. By Lemma~\ref{lem:local}$(i)$, the faces at a vertex $h$ of a $k$-angulation have this incidence structure. We write $\Phi_h$ for the corresponding functional with $\tau=0$ on $\NH(h)$. Put
\begin{equation}\label{eq:ck}
c_k=4^{1/k}(k-1)^{-(k-1)/k},\qquad \gamma_k=\frac{c_k}{4k}.
\end{equation}

\begin{lem}\label{lem:fan}
For every integer $d\ge1$ and $W>0$:
\begin{itemize}
\item[$(i)$] $\Theta_d(W)\le c_k\,d^{1/k}W^{(k-1)/k}$;
\item[$(ii)$] $\Theta_{d+1}(W)-\Theta_d(W)\ge\gamma_k(d+1)^{-(k-1)/k}W^{(k-1)/k}$.
\end{itemize}
For $k=3$ and $\tau\ge0$, the following bound holds for every $d\ge1$ and $W>0$:
\begin{itemize}
\item[$(iii)$] $d\bigl(W/(d+1)\bigr)^{2/3}\le\Theta_{d,\tau}(W)\le(d+1)^{1/3}W^{2/3}+2\tau W^{1/3}$.
\end{itemize}
Still for $k=3$, as $d\to\infty$, the following estimates hold uniformly when $\tau$ and $W$ range over fixed compact subintervals of $(0,\infty)$:
\begin{itemize}
\item[$(iv)$] With $\gamma=(2^{1/3}-2^{-1/3})^{2}$,
\[
\Theta_{d+1,\tau}(W)-\Theta_{d,\tau}(W)\ge\gamma W^{2/3}d^{-2/3}-O(d^{-1});
\]
\item[$(v)$] uniformly also for $0\le M\le W/2$,
\[
\Theta_{d,\tau}(W)-\Theta_{d,\tau}(W-M)\ge\Bigl(\frac23-O(d^{-1/3})\Bigr)d^{1/3}W^{-1/3}M.
\]
\end{itemize}
\end{lem}

\begin{proof}
$(i)$ Let $y$ attain $\Theta_d(W)$. Call $F_i$ \emph{live} if $y>0$ on $F_i\setminus\{h\}$. Different maximal runs of live faces have disjoint non-hub vertex sets. The restriction of $y$ to a run with $d_R$ faces and mass $W_R$ is optimal for that run. By H\"older's inequality, it suffices to bound its contribution by $c_kd_R^{1/k}W_R^{(k-1)/k}$. Relabel the run as $F_1,\dots,F_d$, with mass $W$. Lagrange multipliers give $\partial\Phi/\partial y_v=\mu ky_v^{k-1}$ on the run. Put $p_i=\prod_{v\in F_i\setminus\{h\}}y_v$ and $p_0=p_{d+1}=0$. Multiplication by $y_v$ gives
\begin{equation}\label{eq:lagr}
p_i=\mu ky_v^{k}\ (v\in I_i),\qquad p_j+p_{j+1}=\mu ky_{u_j}^{k}\ (0\le j\le d),
\end{equation}
Thus $\mu>0$. Taking the product over $F_i\setminus\{h\}$ gives $p_i^{3}=c(p_{i-1}+p_i)(p_i+p_{i+1})$, where $c=(\mu k)^{-(k-1)}$; this relation is independent of $k$. Put $S=\sum_ip_i$ and $q_i=(p_{i-1}+2p_i+p_{i+1})/2$. Then $p_i\le c^{1/3}q_i^{2/3}$ and $\sum_iq_i\le2S$, so H\"older's inequality gives $S\le c^{1/3}d^{1/3}(2S)^{2/3}$, hence $S\le4cd$. Summing~\eqref{eq:lagr} over the variables gives $\mu kW=(k-3)S+2S$. Substituting $c=(W/((k-1)S))^{k-1}$ into $S\le4cd$ yields $S^{k}\le4dW^{k-1}(k-1)^{-(k-1)}$.

$(ii)$ Let $y^1,y^2$ attain $\Theta_{d_1}(W_1)$, $\Theta_{d_2}(W_2)$. Identify the last spoke of the first fan with the first spoke of the second, assigning the larger of their two coordinates to the common vertex. The resulting fan has $d_1+d_2$ faces, mass at most $W_1+W_2$, and no smaller face products. Therefore
\[
\Theta_{d_1+d_2}(W_1+W_2)\ge\theta_{d_1}W_1^{(k-1)/k}+\theta_{d_2}W_2^{(k-1)/k}.
\]
Taking $W_i=\theta_{d_i}^{k}/(\theta_{d_1}^{k}+\theta_{d_2}^{k})$ gives $\theta_{d_1+d_2}^{k}\ge\theta_{d_1}^{k}+\theta_{d_2}^{k}$. Since $\theta_1=(k-1)^{-(k-1)/k}$, part $(i)$ and the inequality $\theta_{d+1}^k-\theta_d^k\le k\theta_{d+1}^{k-1}(\theta_{d+1}-\theta_d)$ imply
\[
\theta_{d+1}-\theta_d\ge\frac{\theta_1^{k}}{k\theta_{d+1}^{k-1}}\ge\frac{(k-1)^{-(k-1)}}{kc_k^{k-1}(d+1)^{(k-1)/k}}=\gamma_k(d+1)^{-(k-1)/k} .
\]

$(iii)$ Use the constant vector for the lower bound. For the upper bound, apply $2uv\le u^2+v^2$ along the path and then H\"older's inequality.

$(iv)$ It suffices to consider $d\ge2$. Let $y=(y_0,\dots,y_d)$ attain $\Theta_{d,\tau}(W)$, and put $Q=\sum_{i=1}^dy_{i-1}y_i$, $\Lambda=\sum_{i=0}^dy_i^2$ and $\omega=2^{-1/3}$. Replace an interior coordinate $y_i$ by two consecutive entries $\omega y_i$. This preserves the mass, since $2\omega^3=1$, and changes the functional by $\Delta_i=\omega^2y_i^2-(1-\omega)y_i(y_{i-1}+y_{i+1})$. Since $\sum_{i=1}^{d-1}y_i(y_{i-1}+y_{i+1})\le2Q\le2\Lambda-y_0^2-y_d^2$,
\[
\sum_{i=1}^{d-1}\Delta_i\ge\omega^2(\Lambda-y_0^2-y_d^2)-2(1-\omega)Q\ge\gamma\Lambda-K_1,
\]
where $\gamma=\omega^2-2(1-\omega)$ and the endpoint terms are bounded because $y_0,y_d\le W^{1/3}$. By $(iii)$, $\Lambda\ge Q\ge\Theta_{d,\tau}(W)-2\tau W^{1/3}\ge d^{1/3}W^{2/3}-K_2$. At least one of the $d-1$ increments is at least their average, which proves $(iv)$. The constants are uniform over the stated compact intervals.

$(v)$ Let $y$ attain $\Theta_{d,\tau}(W-M)$, and put $c=(W/(W-M))^{1/3}$. Scaling to $cy$ gives mass $W$ and increases the functional by at least $(c^2-1)Q$, where $Q$ is the path contribution of $y$. As in $(iv)$, $Q\ge d^{1/3}(W-M)^{2/3}-K$. Also
\[
(c^2-1)(W-M)^{2/3}=W^{2/3}-(W-M)^{2/3}\ge\frac23W^{-1/3}M.
\]
Since $M\le W/2$ and $W$ is bounded away from zero, $c^2-1=O(M)$ uniformly. This gives $(v)$.
\end{proof}

\begin{figure}[H]
\centering
\begin{tikzpicture}[scale=1.0,
spoke/.style={circle,fill=black,inner sep=1.5pt},
inter/.style={circle,draw=black,fill=white,inner sep=1.5pt},
hub/.style={circle,fill=black,inner sep=2.6pt}]
\def\Rad{3.2}
\foreach \i/\st in {0/spoke, 1/inter, 2/inter, 3/spoke, 4/inter, 5/inter, 6/spoke, 7/inter, 8/inter, 9/spoke}
{\pgfmathsetmacro{\ang}{180-\i*20}
\node[\st] (v\i) at (\ang:\Rad) {};}
\node[hub] (h) at (0,-0.55) {};
\node[below] at (h) {$h$};
\node[above left] at (v0) {$u_0$};
\node[above] at (v3) {$u_1$};
\node[above] at (v6) {$u_2$};
\node[above right] at (v9) {$u_3$};
\foreach \i in {0,...,8} {\pgfmathtruncatemacro{\j}{\i+1} \draw (v\i)--(v\j);}
\foreach \i in {0,3,6,9} {\draw (h)--(v\i);}
\node at (-1.55,1.15) {$F_1$}; \node at (0,1.35) {$F_2$}; \node at (1.55,1.15) {$F_3$};
\end{tikzpicture}
\caption{The fan $k$-angulation $\FF_{k,f}$ for $k=5$ and $f=3$, with spokes solid and interior vertices hollow.}\label{fig:fan}
\end{figure}

The three candidate hypergraphs give the lower bounds used in the extremal comparisons.

\begin{prop}\label{prop:values}
\begin{itemize}
\item[$(i)$] For $k\ge3$ and $f\ge1$, $\lambda(\HH(\FF_{k,f}))^{k}\ge4f^{3}/(f+1)^{2}\ge4f-8$.
\item[$(ii)$] For a balanced theta graph with $t$ faces, $\lambda(\HH(\Theta(p_1,\dots,p_t)))=2^{1-2/k}t^{2/k}$. More precisely, if $\pi_i$ is the product of the coordinates on the internal vertices of the $i$th path and $W$ is the total mass on all internal vertices, then
\[
\sum_i\pi_i\pi_{i+1}\le\Xi(t,W):=t^{2/k}\Bigl(\frac{2W}{k-2}\Bigr)^{(k-2)/k},
\]
with equality when the internal coordinates are constant.
\item[$(iii)$] $\lambda(D_N)\ge2^{4/3}N^{1/3}-2^{4/3}N^{-2/3}$.
\end{itemize}
\end{prop}

\begin{proof}
$(i)$ Assign $x_h^k=1/k$ to the hub, $x_{u_j}^k=2/(k(f+1))$ to the spokes and $x_v^k=1/(kf)$ to the interior vertices. The mass is $1/k+2/k+(k-3)/k=1$. Every face has a hub, two spokes and $k-3$ interior vertices, so $P(x)=kf\,x_hx_u^2x_v^{k-3}=4^{1/k}f^{3/k}(f+1)^{-2/k}$.

$(ii)$ Let $m_i$ be the mass of the $i$th path and $A,B$ the branch coordinates. The $i$th face has $k-2$ internal vertices, so $(\pi_i\pi_{i+1})^{k/(k-2)}\le(m_i+m_{i+1})/(k-2)$. H\"older's inequality and $\sum_i(m_i+m_{i+1})=2W$ give the asserted bound, with equality for constant internal coordinates because all $m_i+m_{i+1}$ are then equal. Hence $P(x)\le kAB\,\Xi(t,W)$. Put $z=(A^k+B^k)/2$, so $AB\le z^{2/k}$ and $W=1-2z$. The expression $z^2(1-2z)^{k-2}$ is maximized only at $z=1/k$. This yields $2^{1-2/k}t^{2/k}$, attained when $A=B=k^{-1/k}$ and the internal coordinates are constant.

$(iii)$ Assign $6^{-1/3}$ to both hubs and $(2/(3N))^{1/3}$ to each path vertex. The vector has mass one, and its value is $P_{D_N}(x)=2^{2/3}N^{-1/3}+2^{4/3}N^{1/3}(1-1/N)$.
\end{proof}

\begin{lem}\label{lem:bip}
\begin{itemize}
\item[$(i)$] Every closed $k$-angulation $G$ satisfies $\lambda(\HH(G))\le\lambda_1(G)$.
\item[$(ii)$] Every planar bipartite graph $G$ on $n\ge4$ vertices satisfies $\lambda_1(G)\le\sqrt{2n-4}$, with equality if and only if $G\cong K_{2,n-2}$.
\end{itemize}
\end{lem}

\begin{proof}
$(i)$ Put $y_v=x_v^{k/2}$. For a face with boundary cycle $v_1\cdots v_k$, the product $\prod_ix_{v_i}$ is the geometric mean of the $k$ numbers $y_{v_i}y_{v_{i+1}}$. Every edge lies on two faces, so the arithmetic--geometric mean inequality gives $P_{\HH(G)}(x)\le2\sum_{uv\in E(G)}y_uy_v\le\lambda_1(G)\sum_vy_v^2$.

$(ii)$ Let $A,B$ be the parts. The largest row sum of $A(G)^2$, namely $\max_u\sum_{v\sim u}\deg_G(v)$, bounds $\lambda_1(G)^2$. For $u\in A$, this sum counts the edges between $N(u)\subseteq B$ and $A$. If $|N(u)|+|A|\ge3$, planarity bounds it by $2(|N(u)|+|A|)-4\le2n-4$; otherwise it is at most one. If equality holds, $G$ is connected, since joining components would increase $\lambda_1$ within the class. The restriction of $A(G)^2$ to the part attaining the maximum is irreducible. All its row sums must therefore equal $2n-4$, forcing $N(u)=B$ for every $u$ in that part. Thus $G=K_{|A|,|B|}$. Nonplanarity of $K_{3,3}$ excludes $\min\{|A|,|B|\}\ge3$, and equality for a star would give $n-1=2n-4$, contrary to $n\ge4$. Hence $G=K_{2,n-2}$, which attains equality.
\end{proof}

\section{Main Results}\label{sec:main}

\subsection{The outerplanar problem}\label{sec:outer}

Let $G$ be a $k$-angulation with $f\ge2$ faces and $n=(k-2)f+2$ vertices. Assume throughout this subsection that $\HH(G)$ has maximum spectral radius among all $k$-angulation hypergraphs on $n$ vertices. Write $\HH=\HH(G)$ and $\lambda=\lambda(\HH)$, and let $x$ be its Perron vector normalized by $\sum_vx_v^{k}=1$, with coordinates $x_1\ge x_2\ge\cdots\ge x_n$. Let $h_1$ carry $x_1$, and put $d_1=d(h_1)$. Then $x_v\le1$ and
\begin{equation}\label{eq:total}
\frac{\lambda}{k}=\sum_F\ \prod_{v\in F}x_v .
\end{equation}
We use $c_k\le1$ and $4^{1/k}/k\le1$ for $k\ge3$.

In the fan, the hub carries mass $1/k$, and Lemma~\ref{lem:fan}$(i)$ is asymptotically sharp on the remaining vertices. We first show that an extremal hypergraph has the same limiting distribution at its largest coordinate. Fix $\varepsilon\in(0,1)$, let $L=\{v:x_v\ge\varepsilon\}$ be the \emph{heavy} vertices, and put $Z_\varepsilon=\sum_{v\notin L}x_v^{k}$. Then $|L|\le\varepsilon^{-k}$. A light vertex is \emph{ambiguous} if it lies on faces with two distinct heavy vertices, and is \emph{attached to $h\in L$} if $h$ is the only heavy vertex sharing a face with it. Let $B$ be the ambiguous vertices and $N_h$ the light vertices attached to $h$. The sets $N_h$ are disjoint subsets of $\NH(h)$, and Lemma~\ref{lem:local}$(iii)$ gives $|B|\le2k\binom{|L|}2\le k\varepsilon^{-2k}$.

\setcounter{claim}{0}
\begin{claim}\label{cl:master}
With $C_k=2^{3/k}(k-2)^{(3-k)/k}$,
\[
\frac{\lambda}{kf^{1/k}}\le c_kx_1Z_\varepsilon^{(k-1)/k}+C_k\varepsilon+\frac{5k\varepsilon^{-3k}}{f^{1/k}}.
\]
If $L=\{h_1\}$, then
\[
\frac{\lambda}{kf^{1/k}}\le c_kx_1(d_1/f)^{1/k}Z_\varepsilon^{(k-1)/k}+C_k\varepsilon+\frac{\varepsilon^{k}}{f^{1/k}}.
\]
\end{claim}

\begin{proof}
Partition the faces into $\FF_0$, those with no heavy vertex; $\FF_{\ge2}$, those with at least two; $\FF_B$, those with exactly one heavy vertex and an ambiguous vertex; and $\FF_1$, the remaining faces. By~\eqref{eq:facesout}, with $W\le1$, the weight of $\FF_0$ is at most $2^{2/k}(k-2)^{-(k-2)/k}\varepsilon n^{1/k}+\varepsilon^k$. A heavy pair lies on at most two faces, so $|\FF_{\ge2}|\le\varepsilon^{-2k}$. Similarly, $|\FF_B|\le2|L||B|\le2k\varepsilon^{-3k}$, and each face has weight at most one.

For $h\in L$, let $y^{(h)}$ agree with $x$ on $N_h$ and vanish elsewhere on $\NH(h)$, and put $Z_h=\sum_{v\in N_h}x_v^k$. Then $x_h\Phi_h(y^{(h)})$ is the weight of the faces of $\FF_1$ at $h$. Lemma~\ref{lem:fan}$(i)$ and H\"older's inequality give
\[
\sum_{F\in\FF_1}\prod_{v\in F}x_v\le c_k\sum_{h\in L}x_hd(h)^{1/k}Z_h^{(k-1)/k}\le c_kx_1\Bigl(\sum_{h\in L}d(h)\Bigr)^{1/k}Z_\varepsilon^{(k-1)/k}.
\]
Here $\sum_hZ_h\le Z_\varepsilon$ and $\sum_{h\in L}d(h)=\sum_F|F\cap L|\le f+(k-1)|\FF_{\ge2}|\le f+k\varepsilon^{-2k}$. Divide~\eqref{eq:total} by $f^{1/k}$. Since $n\le2(k-2)f$, the contribution of $\FF_0$ is at most $C_k\varepsilon+\varepsilon^kf^{-1/k}$. Using $(1+t)^{1/k}\le1+t/k$ and $c_kx_1Z_\varepsilon^{(k-1)/k}\le1$, the main term is at most $c_kx_1Z_\varepsilon^{(k-1)/k}+\varepsilon^{-2k}/f$. The errors total at most $(3+2k)\varepsilon^{-3k}f^{-1/k}\le5k\varepsilon^{-3k}f^{-1/k}$. If $L=\{h_1\}$, then $\FF_{\ge2}=\FF_B=\emptyset$ and $\sum_{h\in L}d(h)=d_1$.
\end{proof}

The function $g(z)=z(1-z)^{k-1}$ is maximized on $[0,1]$ only at $z=1/k$, where its value is $(k-1)^{k-1}k^{-k}$. Consequently, for $p\in[0,1]$,
\begin{equation}\label{eq:meet}
c_kp(1-p^k)^{(k-1)/k}=4^{1/k}(k-1)^{-(k-1)/k}g(p^k)^{1/k}\le\frac{4^{1/k}}{k},
\end{equation}
with equality only when $p^k=1/k$. Extremality and Proposition~\ref{prop:values}$(i)$ give the matching lower bound $\lambda^k\ge4f-8$.

\begin{claim}\label{cl:stab}
Along every sequence of extremal $k$-angulation hypergraphs with $f\to\infty$,
\[
x_1^{k}\to\frac1k,\qquad x_2\to0,\qquad \frac{d_1}{f}\to1,\qquad\frac{\lambda^k}{4f}\to1 .
\]
\end{claim}

\begin{proof}
It suffices to prove these limits along a subsequence of any given subsequence. Assume $x_1\to p$, $x_2\to q$ and $d_1/f\to\sigma$. The lower bound gives $\liminf\lambda/(kf^{1/k})\ge4^{1/k}/k$. If $p=0$, Claim~\ref{cl:master} with $Z_\varepsilon\le1$ gives $\limsup\lambda/(kf^{1/k})\le C_k\varepsilon$ for every fixed $\varepsilon$, a contradiction. Thus $p>0$. For $\varepsilon<p/2$, $h_1$ is eventually heavy and $Z_\varepsilon\le1-x_1^k$. Letting $f\to\infty$ and then $\varepsilon\to0$, Claim~\ref{cl:master} and~\eqref{eq:meet} give
\[
\frac{4^{1/k}}{k}\le\liminf\frac{\lambda}{kf^{1/k}}\le\limsup\frac{\lambda}{kf^{1/k}}\le c_kp(1-p^k)^{(k-1)/k}\le\frac{4^{1/k}}{k}.
\]
Hence $p^k=1/k$ and $\lambda^k/(4f)\to1$. If $q>0$, take $\varepsilon<q/2$; then $Z_\varepsilon\le1-x_1^k-x_2^k$, and the same limit gives $4^{1/k}/k\le c_kp(1-p^k-q^k)^{(k-1)/k}<4^{1/k}/k$. Thus $q=0$. For fixed $\varepsilon<p/2$, eventually $L=\{h_1\}$. The second bound of Claim~\ref{cl:master} yields $4^{1/k}/k\le\sigma^{1/k}4^{1/k}/k+C_k\varepsilon$. Letting $\varepsilon\to0$ gives $\sigma=1$.
\end{proof}

We next exclude faces avoiding $h_1$. Write $A=x_1$, let $\FA$ be the family of these faces, and put $a=|\FA|=f-d_1$. Root $T$ at a face containing $h_1$. The faces at $h_1$ form a path $P_{h_1}$ in $T$, and $\FA$ is the union of the components of $T-P_{h_1}$, each attached through one chord. Put
\[
U=V(G)\setminus(\{h_1\}\cup\NH(h_1)),\qquad u=\sum_{v\in U}x_v^{k},\qquad W_1=\sum_{v\in\NH(h_1)}x_v^{k},
\]
Thus $A^k+W_1+u=1$. Writing $\Phi=\Phi_{h_1}$, we have
\begin{equation}\label{eq:split}
\frac{\lambda}{k}=A\,\Phi(x)+S_0,\qquad S_0=\sum_{F\in\FA}\ \prod_{v\in F}x_v .
\end{equation}

\begin{claim}\label{cl:own}
$U=\bigcup_{F\in\FA}O(F)$, and at most $2a$ vertices of $\NH(h_1)$ lie on faces of $\FA$.
\end{claim}

\begin{proof}
Let $F\in\FA$, $v\in O(F)$, and suppose $v$ lies on a face $F'$ with $h_1$. As in Lemma~\ref{lem:local}$(iv)$, $F$ is the face at $v$ nearest the root, so the path from $F'$ to the root passes through $F$. This path stays in $P_{h_1}$, contradicting $F\in\FA$. The faces at $h_1$ own only vertices of $\{h_1\}\cup\NH(h_1)$, so the first assertion follows from the partition in Lemma~\ref{lem:local}$(iv)$. If $F$ lies in a component attached through a chord $c$ and $w\in F\cap\NH(h_1)$, every chord on the path from $F$ to a face at $h_1$ containing $w$ contains $w$. This path uses $c$, so $w\in c$. There are at most $a$ components.
\end{proof}

\begin{claim}\label{cl:compare}
$S_0\ge R:=A\bigl(\Theta_f(1-A^{k})-\Phi(x)\bigr)$.
\end{claim}

\begin{proof}
Assign $A$ to the hub of $\FF_{k,f}$ and a maximizer of $\Theta_f(1-A^k)$ to its other vertices. This gives a vector of mass one. Extremality yields $\lambda\ge kA\,\Theta_f(1-A^k)$; compare with~\eqref{eq:split}.
\end{proof}

Fix $K\ge1$, to be chosen in terms of $k$, and put $Q=\{v\in\NH(h_1):x_v^k\ge K/f\}$ and $m_Q=\sum_{v\in Q}x_v^k$. Then $|Q|\le fm_Q/K$. The next estimate accounts for both the mass $u$ outside the hub fan and the mass $m_Q$ concentrated on its high-coordinate vertices.

\begin{claim}\label{cl:loss}
\[
\begin{aligned}
\frac RA&\ge\gamma_k(1-A^k)^{(k-1)/k}af^{-(k-1)/k}+\frac{k-1}{k}\theta_{d_1}(1-A^k)^{-1/k}(u+m_Q)\\
&\qquad-c_k2^{1/k}(2k-3)^{(k-1)/k}(f/K)^{1/k}m_Q .
\end{aligned}
\]
\end{claim}

\begin{proof}
Let $x'$ agree with $x$ on $\NH(h_1)\setminus Q$ and vanish on $Q$. Let $\mathcal S$ be the faces at $h_1$ meeting $Q$, and write $P_Q=\Phi(x)-\Phi(x')$ for their contribution. Every non-hub vertex lies on at most two faces at $h_1$, so $|\mathcal S|\le2|Q|$. The mass $m_{\mathcal S}$ on their non-hub vertices is at most $m_Q+2|Q|(k-2)K/f\le(2k-3)m_Q$. Split $\mathcal S$ into maximal consecutive runs. Their non-hub vertex sets are disjoint by Lemma~\ref{lem:local}$(ii)$, so Lemma~\ref{lem:fan}$(i)$ and H\"older's inequality give
\[
P_Q\le c_k|\mathcal S|^{1/k}m_{\mathcal S}^{(k-1)/k}\le c_k2^{1/k}(2k-3)^{(k-1)/k}(f/K)^{1/k}m_Q.
\]
Since $\Phi(x')\le\Theta_{d_1}(W_1-m_Q)$,
\[
\Theta_f(1-A^k)-\Phi(x)\ge\bigl[\Theta_f(1-A^k)-\Theta_{d_1}(1-A^k)\bigr]+\bigl[\Theta_{d_1}(1-A^k)-\Theta_{d_1}(W_1-m_Q)\bigr]-P_Q .
\]
By Lemma~\ref{lem:fan}$(ii)$, the first bracket is at least $\gamma_k(1-A^k)^{(k-1)/k}af^{-(k-1)/k}$. Since $W_1-m_Q=1-A^k-u-m_Q$, concavity of $s\mapsto s^{(k-1)/k}$ bounds the second below by $\theta_{d_1}(k-1)k^{-1}(1-A^k)^{-1/k}(u+m_Q)$.
\end{proof}

\begin{claim}\label{cl:gain}
$S_0\le(k-2)^{-(k-2)/k}(a/2)^{1/k}(u+m_Q+2Ka/f)$.
\end{claim}

\begin{proof}
Apply~\eqref{eq:facesout} to $\FA$, which omits a face. By Claim~\ref{cl:own}, its vertices lie in $U$ and at most $2a$ vertices of $\NH(h_1)$. Those in $Q$ have total mass at most $m_Q$, and each of the others has $x_v^k<K/f$.
\end{proof}

\begin{proof}[Proof of Theorem~\ref{thm:main}$(a)$]
An extremal hypergraph exists because the class is finite, and Claim~\ref{cl:stab} gives the asymptotic maximum. Suppose there is a sequence of extremal hypergraphs with $f\to\infty$ and $a\ge1$. For large $f$, Claim~\ref{cl:stab} gives $a\le f/2$, $1/(2k)\le A^k\le1/2$ and $a/f\to0$. Thus $d_1\ge f/2$. On a fan with $d_1$ faces, assign $y_v^k=2c$ to the spokes and $y_v^k=c$ to the interior vertices, where $c=(d_1(k-1)+2)^{-1}$. This gives
\[
\theta_{d_1}\ge4^{1/k}d_1\bigl(d_1(k-1)+2\bigr)^{-(k-1)/k}\ge\frac{c_k}2d_1^{1/k}\ge\frac{c_k}4f^{1/k}.
\]
Also $A\ge1/2$, $(1-A^k)^{(k-1)/k}\ge1/2$ and $(1-A^k)^{-1/k}\ge1$. Put $c_M=(k-1)c_k/(4k)$ and choose $K$ so large that $c_k2^{1/k}(2k-3)^{(k-1)/k}K^{-1/k}\le c_M/2$. Claim~\ref{cl:loss} then gives
\[
R\ge\frac{\gamma_k}4af^{-(k-1)/k}+\frac{c_M}4f^{1/k}(u+m_Q).
\]
On the other hand, Claim~\ref{cl:gain}, with $C_0=(k-2)^{-(k-2)/k}2^{-1/k}$, gives
\[
S_0\le C_0(a/f)^{1/k}f^{1/k}(u+m_Q)+2C_0K(a/f)^{1/k}af^{-(k-1)/k} .
\]
Since $a/f\to0$, eventually $C_0(a/f)^{1/k}\le c_M/8$ and $2C_0K(a/f)^{1/k}\le\gamma_k/8$. Hence $S_0\le R/2<R$, since $a\ge1$. This contradicts Claim~\ref{cl:compare}. Every face of an extremal hypergraph therefore contains $h_1$ for large $f$, and Lemma~\ref{lem:local}$(i)$ identifies $G$ as the fan $k$-angulation.
\end{proof}

\subsection{The planar problem for $k\ge4$}\label{sec:pk4}

For $k=4$, the comparison with the adjacency spectral radius suffices.

\begin{proof}[Proof of Theorem~\ref{thm:main}$(c)$ for $k=4$]
Let $\HH=\HH(G)$, where $G$ is a closed $4$-angulation on $n\ge5$ vertices. A plane graph whose faces all have even length is bipartite, so Lemma~\ref{lem:bip} gives $\lambda(\HH)\le\lambda_1(G)\le\sqrt{2n-4}$, with equality only if $G=K_{2,n-2}$. Conversely, $K_{2,n-2}$ is a balanced theta graph with $t=n-2\ge3$ and every $p_i=1$. Its face hypergraph has spectral radius $\sqrt{2(n-2)}$ by Proposition~\ref{prop:values}$(ii)$.
\end{proof}

For $k\ge5$, the estimate in Lemma~\ref{lem:bip}$(i)$ is strict at balanced theta graphs: each face has an edge between two internal vertices, whose coordinates differ from the branch coordinates. We therefore compare the face functionals directly. The argument below is valid for every $k\ge4$.

Let $G$ be a closed $k$-angulation on $n>k$ vertices with $f=2(n-2)/(k-2)$ faces, and suppose $\HH=\HH(G)$ is extremal. Since $n>k$, distinct faces are distinct hyperedges (Section~\ref{sec:intro}). Let $x$, $\lambda$ and $x_1\ge x_2\ge\cdots\ge x_n$ be as in Subsection~\ref{sec:outer}; then $\lambda/k=\sum_F\prod_{v\in F}x_v$. Extremality and Proposition~\ref{prop:values}$(ii)$ give $\lambda\ge2^{1-2/k}f^{2/k}$. In a balanced theta graph, each face contains both branch vertices, while no third vertex can share more than two faces with them by Lemma~\ref{lem:three}. Fix $\varepsilon\in(0,1)$, let $L=\{v:x_v\ge\varepsilon\}$ and $Z=\sum_{v\notin L}x_v^{k}$, and partition the faces into $\FF_j$ $(j=0,1,2)$ and $\FF_{\ge3}$ according to the number of heavy vertices. Let their total weights be $\Sigma_j$ and $\Sigma_{\ge3}$.

\setcounter{claim}{0}
\begin{claim}\label{cl:pcrude}
With $C_1=(6/(k-2))^{(k-2)/k}$,
\[
\frac\lambda k\le C_1f^{2/k}\bigl(x_1x_2+x_1\varepsilon+\varepsilon^{2}\bigr)+\frac{k+3}{3k}\varepsilon^{-2k}.
\]
More precisely, $\Sigma_2\le C_1f^{2/k}x_1x_2$ and $\Sigma_0+\Sigma_1\le C_1f^{2/k}(x_1\varepsilon+\varepsilon^2)$.
\end{claim}

\begin{proof}
By Lemma~\ref{lem:three}, a heavy triple lies on at most two faces. Hence $|\FF_{\ge3}|\le2\binom{|L|}3$, and each heavy vertex lies on at most $|L|^{2}$ faces of this family. The arithmetic--geometric mean inequality gives
\[
\Sigma_{\ge3}\le\frac1k\sum_{F\in\FF_{\ge3}}\sum_{v\in F}x_v^{k}\le\frac1k\Bigl(|L|^{2}\sum_{h\in L}x_h^{k}+\varepsilon^{k}k|\FF_{\ge3}|\Bigr)\le\frac{k+3}{3k}\varepsilon^{-2k}.
\]
For each family $\FF_j$, $j\le2$, take the sets $\phi(F)$ of Lemma~\ref{lem:assign}$(iii)$ with $H_0=\emptyset$. Write $F\setminus\phi(F)=\{a_F,b_F\}$ and $Y_F=\prod_{v\in\phi(F)}x_v$. H\"older's inequality, the arithmetic--geometric mean inequality and the multiplicity bound give $\sum_FY_F\le f^{2/k}(6/(k-2))^{(k-2)/k}$. The remaining product $x_{a_F}x_{b_F}$ is at most $\varepsilon^{2}$ on $\FF_0$, at most $x_1\varepsilon$ on $\FF_1$, and at most $x_1x_2$ on $\FF_2$.
\end{proof}

Divide by $f^{2/k}$ and use $\lambda\ge2^{1-2/k}f^{2/k}$. A fixed sufficiently small $\varepsilon$ gives a constant $\eta=\eta(k)>0$ such that $x_2\ge x_1x_2\ge\eta$ for all large $n$. The constant $C_1$ exceeds the sharp constant $(2/(k-2))^{(k-2)/k}$ of the candidate, so this bound does not determine the limits of $x_1,x_2$. To recover that constant, we separate the light vertices serving different heavy pairs. Lemma~\ref{lem:three} bounds the multiplicity within a pair by two, and Lemma~\ref{lem:amb} controls the overlap between pairs.

From now on, take $0<\varepsilon<\eta$. Let $h_1\ne h_2$ carry $x_1,x_2$, and put $e_0=\{h_1,h_2\}$; thus $h_1,h_2\in L$. A light vertex is \emph{ambiguous} if it shares faces with at least three distinct heavy vertices. Such a vertex either lies on a face of $\FF_{\ge3}$ or belongs to the set $S$ in Lemma~\ref{lem:amb} for a heavy triple. Hence the set $S^*$ of ambiguous vertices satisfies
\[
|S^*|\le2k\binom{|L|}3+(6k-10)\binom{|L|}3\le2k\varepsilon^{-3k}.
\]
Split $\FF_2$ into $\FF_2^0$, the faces containing $e_0$; $\FF_2^*$, the faces not containing $e_0$ that contain an ambiguous vertex; and $\FF_2'$, the rest. Put $t_0=|\FF_2^0|$ and $s=t_0/f$.

\begin{claim}\label{cl:psharp}
\[
\Sigma_2\le x_1\bigl(x_2^{k/2}t_0+x_3^{k/2}(f-t_0)\bigr)^{2/k}\Bigl(\frac{2Z}{k-2}\Bigr)^{(k-2)/k}+2k\varepsilon^{-5k}.
\]
\end{claim}

\begin{proof}
For a heavy pair $e$, let $V_e$ be the light vertices on faces of $\FF_2^0$ if $e=e_0$, and on faces of $\FF_2'$ with heavy pair $e$ otherwise. Write $Z_e$ for their mass. A light vertex on a face of $\FF_2'$ with heavy pair $e$ is not ambiguous, so the only heavy vertices sharing a face with it are those of $e$. Thus the sets $V_e$ are disjoint and $\sum_eZ_e\le Z$. Every face of $\FF_2^*$ contains a heavy pair and an ambiguous vertex. By Lemma~\ref{lem:three}, each such triple lies on at most two faces, so $|\FF_2^*|\le|S^*||L|^{2}\le2k\varepsilon^{-5k}$.

Let $t_e$ count the faces of $\FF_2^0$, respectively $\FF_2'$, with heavy pair $e$. A light vertex lies on at most two faces containing $e$, so Lemma~\ref{lem:pair} bounds the sum of their light products by $t_e^{2/k}(2Z_e/(k-2))^{(k-2)/k}$. The heavy product is $x_1x_2$ for $e_0$ and at most $x_1x_3$ for any other pair. Apply H\"older's inequality with exponents $k/2$ and $k/(k-2)$ first over $e\ne e_0$, using $\sum_{e\ne e_0}t_e\le f-t_0$, and then to the two resulting terms.
\end{proof}

The optimization corresponding to~\eqref{eq:meet} is
\begin{equation}\label{eq:psi}
\Psi(p,q):=pq\Bigl(\frac{2(1-p^{k}-q^{k})}{k-2}\Bigr)^{(k-2)/k}\le\frac{2^{1-2/k}}{k}\qquad(p,q\ge0,\ p^k+q^k\le1),
\end{equation}
with equality only at $p=q=k^{-1/k}$. Indeed, putting $z=(p^k+q^k)/2$ gives $pq\le z^{2/k}$, with equality only for $p=q$. The derivative of $z^{2}(1-2z)^{k-2}$ on $[0,1/2]$ is $2z(1-2z)^{k-3}(1-kz)$, so its unique maximum occurs at $z=1/k$.

\begin{claim}\label{cl:pstab}
Along every sequence of extremal planar $k$-uniform hypergraphs with $n\to\infty$,
\[
x_1^{k},x_2^{k}\to\frac1k,\qquad x_3\to0,\qquad\frac{\lambda}{f^{2/k}}\to2^{1-2/k},\qquad\frac tf\to1,
\]
where $t$ is the number of faces containing both $h_1,h_2$.
\end{claim}

\begin{proof}
Add the bounds for $\Sigma_0+\Sigma_1$ and $\Sigma_{\ge3}$ from Claim~\ref{cl:pcrude} to Claim~\ref{cl:psharp}, and divide by $f^{2/k}$. This gives
\begin{equation}\label{eq:pmaster}
\begin{aligned}
\frac{2^{1-2/k}}{k}\le\frac\lambda{kf^{2/k}}&\le x_1\bigl(x_2^{k/2}s+x_3^{k/2}(1-s)\bigr)^{2/k}\Bigl(\frac{2Z}{k-2}\Bigr)^{(k-2)/k}\\
&\qquad+2C_1\varepsilon+o_\varepsilon(1),
\end{aligned}
\end{equation}
where $o_\varepsilon(1)\to0$ as $n\to\infty$ for fixed $\varepsilon$. Pass to a subsequence with $x_1\to p$, $x_2\to q\ge\eta$, $x_3\to r$ and $t/f\to\sigma$. Every face counted by $t$ but not by $t_0$ contains a third heavy vertex $g$, and the triple $\{h_1,h_2,g\}$ lies on at most two faces by Lemma~\ref{lem:three}. Thus $0\le t-t_0\le2\varepsilon^{-k}$, so $s\to\sigma$ for every fixed $\varepsilon<\eta$.

If $r>0$, choose $\varepsilon<\min\{\eta,r/2\}$. Then a vertex carrying $x_3$ is also heavy, $Z\le1-x_1^k-x_2^k-x_3^k$, and the parenthesized factor in~\eqref{eq:pmaster}, raised to $2/k$, is at most $x_2$. Letting $n\to\infty$ and then $\varepsilon\to0$ yields
\[
\frac{2^{1-2/k}}k\le pq\Bigl(\frac{2(1-p^k-q^k-r^k)}{k-2}\Bigr)^{(k-2)/k}<\Psi(p,q),
\]
contrary to~\eqref{eq:psi}. Thus $r=0$. The first factor in~\eqref{eq:pmaster} now tends to $pq\,\sigma^{2/k}$, while $Z\le1-x_1^k-x_2^k$. The same limiting argument gives $2^{1-2/k}/k\le\sigma^{2/k}\Psi(p,q)\le\sigma^{2/k}2^{1-2/k}/k$. Therefore $\sigma=1$ and $p=q=k^{-1/k}$. Finally,~\eqref{eq:pmaster} with $s\le1$ gives the limit of $\lambda/f^{2/k}$.
\end{proof}

We next exclude faces missing either $h_1$ or $h_2$. Put $A=x_1$, $B=x_2$, $\delta=x_3$ and $W=1-A^k-B^k$. For large $f$, Claim~\ref{cl:pstab} gives $A,B\ge1/2$, $1/3\le W\le1$, $\delta\to0$, and $t=f-a$ with $a\le f/2$. Let $\FA$ be the $a$ faces missing at least one of $h_1,h_2$. Let $N_0$ be the vertices other than $h_1,h_2$ on faces containing $e_0$, and let $U$ be the remaining vertices other than $h_1,h_2$. Put $u=\sum_{v\in U}x_v^k$ and $\Phi_{e_0}(x)=\sum_{F\supseteq e_0}\prod_{v\in F\setminus e_0}x_v$. Then
\begin{equation}\label{eq:psplit}
\frac\lambda k=AB\,\Phi_{e_0}(x)+S_0,\qquad S_0=\sum_{F\in\FA}\prod_{v\in F}x_v .
\end{equation}

\begin{claim}\label{cl:pcompare}
$S_0\ge R:=AB\bigl(\Xi(f,W)-\Phi_{e_0}(x)\bigr)$, with $\Xi$ as in Proposition~\ref{prop:values}$(ii)$.
\end{claim}

\begin{proof}
Assign $A,B$ to the branch vertices of a balanced theta graph with $f$ faces and $n$ vertices, and equal coordinates of total mass $W$ to the remaining vertices. This vector has mass one. Proposition~\ref{prop:values}$(ii)$ and extremality give $\lambda/k\ge AB\,\Xi(f,W)$; compare with~\eqref{eq:psplit}.
\end{proof}

Fix $K\ge1$, to be chosen in terms of $k$. Let $Q=\{v\in N_0:x_v^k\ge K/f\}$ and $m_Q=\sum_{v\in Q}x_v^k$, so $|Q|\le fm_Q/K$.

\begin{claim}\label{cl:ploss}
\[
\begin{aligned}
\frac R{AB}&\ge\frac2k\Bigl(\frac{2W}{k-2}\Bigr)^{(k-2)/k}af^{2/k-1}\\
&\qquad+\frac{k-2}{k}\Bigl(\frac2{k-2}\Bigr)^{(k-2)/k}t^{2/k}W^{-2/k}(u+m_Q)-2(f/K)^{2/k}m_Q .
\end{aligned}
\]
\end{claim}

\begin{proof}
Let $x'$ agree with $x$ off $Q$ and vanish on $Q$, and let $\mathcal S$ be the faces containing $e_0$ and a vertex of $Q$. Lemma~\ref{lem:three} gives $|\mathcal S|\le2|Q|$. Count mass with its face multiplicity, as in Lemma~\ref{lem:pair}. The vertices outside $e_0\cup Q$ contribute less than $2(k-3)|Q|K/f\le2(k-3)m_Q$, while those in $Q$ contribute at most $2m_Q$. Thus the mass parameter in that lemma is at most $2(k-2)m_Q$, and
\[
\Phi_{e_0}(x)-\Phi_{e_0}(x')\le(2|Q|)^{2/k}(2m_Q)^{(k-2)/k}\le2(f/K)^{2/k}m_Q.
\]
The vector $x'$ has mass $W-u-m_Q$ on $N_0$, and each such vertex lies on at most two of the $t$ faces containing $e_0$. Hence $\Phi_{e_0}(x')\le\Xi(t,W-u-m_Q)$ by Lemma~\ref{lem:pair}. Finally,
\[
\begin{aligned}
\Xi(f,W)-\Xi(t,W-u-m_Q)&=\bigl(f^{2/k}-t^{2/k}\bigr)\Bigl(\frac{2W}{k-2}\Bigr)^{(k-2)/k}\\
&\qquad+t^{2/k}\Bigl(\frac2{k-2}\Bigr)^{(k-2)/k}\Bigl[W^{(k-2)/k}-(W-u-m_Q)^{(k-2)/k}\Bigr].
\end{aligned}
\]
Concavity of $s\mapsto s^{2/k}$ and $s\mapsto s^{(k-2)/k}$ bounds the two differences below by $(2/k)f^{2/k-1}a$ and $((k-2)/k)W^{-2/k}(u+m_Q)$.
\end{proof}

Since $t\ge f/2$, $A,B\ge1/2$ and $1/3\le W\le1$, choosing $K$ sufficiently large in terms of $k$ gives
\begin{equation}\label{eq:Rp}
R\ge c_La\,f^{2/k-1}+c_Mf^{2/k}(u+m_Q),
\end{equation}
where $c_L,c_M>0$ depend only on $k$. The faces of $\FA$ have $ka$ vertex--face incidences, so the mass $W_{\FA}$ of their vertices other than $h_1,h_2$ satisfies $W_{\FA}\le u+m_Q+kKa/f$.

\begin{claim}\label{cl:pgain}
$S_0\le C_3\delta R$, where $C_3$ depends only on $k$.
\end{claim}

\begin{proof}
Split $\FA$ into $\FA_1$, the faces containing $h_1$ but not $h_2$; $\FA_2$, those containing $h_2$ but not $h_1$; and $\FA_3$, those containing neither. Apply Lemma~\ref{lem:assign}$(iii)$ to $\FA_1$ with $H_0=\{h_1\}$, to $\FA_2$ with $H_0=\{h_2\}$, and to $\FA_3$ with $H_0=\emptyset$. Write $Y_F=\prod_{v\in\phi(F)}x_v$. For $F\in\FA_1$, the remaining vertex other than $h_1$ is not $h_2$, so its coordinate is at most $\delta$ and $\prod_{v\in F}x_v\le\delta Y_F$; the same holds on $\FA_2$, and on $\FA_3$ the factor is at most $\delta^2\le\delta$. Within each family, H\"older's inequality, the arithmetic--geometric mean inequality and the multiplicity bound give $\sum_FY_F\le a^{2/k}(6W_{\FA}/(k-2))^{(k-2)/k}$. Using $(s+s')^{\theta}\le s^\theta+s'^\theta$ for $0<\theta<1$, we obtain
\[
S_0\le3\delta\Bigl(\frac6{k-2}\Bigr)^{(k-2)/k}\Bigl[a^{2/k}(u+m_Q)^{(k-2)/k}+(kK)^{(k-2)/k}af^{-(k-2)/k}\Bigr].
\]
The weighted arithmetic--geometric mean inequality gives
\[
a^{2/k}(u+m_Q)^{(k-2)/k}=\bigl(af^{2/k-1}\bigr)^{2/k}\bigl(f^{2/k}(u+m_Q)\bigr)^{(k-2)/k}\le af^{2/k-1}+f^{2/k}(u+m_Q).
\]
Both terms are controlled by~\eqref{eq:Rp}, as is $af^{-(k-2)/k}=af^{2/k-1}$.
\end{proof}

\begin{proof}[Proof of Theorem~\ref{thm:main}$(c)$ for $k\ge5$]
Since $\delta\to0$ by Claim~\ref{cl:pstab}, Claim~\ref{cl:pgain} gives $S_0\le R/2$ for all large $f$. If $a\ge1$, then~\eqref{eq:Rp} gives $R>0$, contrary to Claim~\ref{cl:pcompare}. Thus every face contains $h_1,h_2$. By Lemma~\ref{lem:three}, every other vertex lies on at most two faces and therefore has degree two. A $2$-connected plane graph in which all vertices but two have degree two consists of internally disjoint paths between those two vertices. Let their numbers of internal vertices be $p_1,\dots,p_f\ge0$ in cyclic order. The face lengths give $p_i+p_{i+1}=k-2$. If some $p_i=0$, then $p_{i+2}=p_i$ forces two paths with no internal vertices, which would be parallel edges. Thus all $p_i\ge1$, and $G$ is a balanced theta graph. Proposition~\ref{prop:values}$(ii)$ gives the same spectral radius for every balanced theta graph with $f$ faces, so each is extremal.

It remains to count the nonisomorphic hypergraphs. The $p_i$ alternate between $p$ and $k-2-p$. If $t=f$ is odd, then $p=(k-2)/2$. If $t$ is even, every unordered pair $\{p,k-2-p\}$ with $1\le p\le k-3$ occurs, giving $\lfloor(k-2)/2\rfloor$ possibilities; exchanging the two values is a cyclic shift. The face hypergraph determines this pair: since $t\ge3$, the two branch vertices form the intersection of all hyperedges, and the intersection of two consecutive hyperedges, after removing the branch vertices, is the internal vertex set of one path.
\end{proof}

\subsection{The planar problem for $k=3$}\label{sec:p3}

Throughout this subsection, $k=3$, $n\ge4$, $N=n-2$ and $\nu=6^{-1/3}$. By Lemma~\ref{lem:augment}, an extremal hypergraph has the form $\HH=\HH(G)$ for a plane triangulation $G$, with $3N$ edges and $2N$ faces. Two vertices share a face precisely when they are adjacent, in which case they share exactly two. Thus no pair can occur on all faces as in a balanced theta graph. In $D_N$, the hubs $a,b$ share only the faces $abu_1$, $abu_N$, and every other face contains exactly one hub. The test vector in Proposition~\ref{prop:values}$(iii)$ assigns $\nu$ to each hub; Claim~\ref{cl:3stab} below shows that the two largest coordinates of an extremal hypergraph tend to $\nu$. We then prove that their vertices are adjacent to each other and to every other vertex. Let $x$, $\lambda$ and $x_1\ge x_2\ge\cdots$ be as before. Proposition~\ref{prop:values}$(iii)$ gives $\lambda\ge2^{4/3}N^{1/3}-o(1)$. All limits are taken as $n\to\infty$ along an arbitrary sequence of extremal hypergraphs. The \emph{link cycle} of a vertex is the cycle of its neighbors in their cyclic order around it, and $N(v)=N_G(v)$ denotes the ordinary neighborhood.

Fix $\varepsilon\in(0,1)$, put $L=\{v:x_v\ge\varepsilon\}$ and $Z=\sum_{v\notin L}x_v^{3}$, and let $\FF_j$ $(j=0,1,2)$ and $\FF_{\ge3}$ be the families of faces with exactly $j$, respectively at least three, vertices in $L$. Write $\Sigma_j$, $\Sigma_{\ge3}$ for their weights. The main contribution now comes from $\FF_1$, for which we need a sharp estimate.

\setcounter{claim}{0}
\begin{claim}\label{cl:3master}
$\lambda/(3n^{1/3})\le(x_1+x_2)Z^{2/3}+200^{1/3}\varepsilon+o_\varepsilon(1)$, where $o_\varepsilon(1)\to0$ as $n\to\infty$ for fixed $\varepsilon$.
\end{claim}

\begin{proof}
By~\eqref{eq:facesp3}, $\Sigma_0\le200^{1/3}\varepsilon n^{1/3}$. A heavy pair lies on at most two faces, so $\Sigma_2+\Sigma_{\ge3}\le|\FF_2|+|\FF_{\ge3}|\le2\varepsilon^{-9}$. For $h\in L$, the light neighbors induce a subgraph $J_h$ of its link cycle with maximum degree at most two. Thus $\sum_{uv\in E(J_h)}x_ux_v\le\sum_{u\in N(h)\setminus L}x_u^{2}$. With $c(u)=\sum_{h\in L\cap N(u)}x_h$, this gives
\[
\Sigma_1=\sum_{h\in L}x_h\sum_{uv\in E(J_h)}x_ux_v\le\sum_{u\notin L}c(u)x_u^{2}.
\]
At most two vertices are adjacent to every vertex of a heavy triple, since three would give a $K_{3,3}$. Hence at most $\varepsilon^{-9}$ light vertices have three or more heavy neighbors, each contributing at most $|L|\le\varepsilon^{-3}$. Every other light vertex has $c(u)\le x_1+x_2$. H\"older's inequality now gives $\Sigma_1\le(x_1+x_2)Z^{2/3}n^{1/3}+\varepsilon^{-12}$.
\end{proof}

\begin{claim}\label{cl:3stab}
$x_1,x_2\to\nu$, $x_3\to0$ and $\lambda^{3}/(16N)\to1$.
\end{claim}

\begin{proof}
Pass to a subsequence with $x_1\to p$, $x_2\to q$, $x_3\to r$. Proposition~\ref{prop:values}$(iii)$ gives $\liminf\lambda/(3n^{1/3})\ge2^{4/3}/3$. If $q=0$, use Claim~\ref{cl:3master}, taking $\varepsilon<p/2$ when $p>0$, and then let $\varepsilon\to0$. This yields $\limsup\lambda/(3n^{1/3})\le p(1-p^3)^{2/3}\le(4/27)^{1/3}<2^{4/3}/3$, since $z(1-z)^2\le4/27$. Thus $q>0$. For $\varepsilon<q/2$, $Z\le1-x_1^3-x_2^3$. Let $n\to\infty$ and then $\varepsilon\to0$ to obtain
\[
\frac{2^{4/3}}{3}\le(p+q)(1-p^{3}-q^{3})^{2/3}.
\]
Since $p^3+q^3\ge(p+q)^3/4$, with equality only for $p=q$, putting $z=(p+q)^3/4$ gives
\[
\frac{16}{27}\le(p+q)^3(1-p^3-q^3)^2\le4z(1-z)^2\le\frac{16}{27}.
\]
Equality throughout forces $p=q$ and $z=1/3$, hence $p=q=\nu$. If $r>0$, the same argument with three heavy vertices gives $2^{4/3}/3\le2\nu(2/3-r^3)^{2/3}<2^{4/3}/3$. Finally, for $\varepsilon<\nu/2$, Claim~\ref{cl:3master} gives the matching upper limit for $\lambda/(3n^{1/3})$.
\end{proof}

From now on, let $a,b$ carry $x_1,x_2$, and put
\[
A=x_a,\quad B=x_b,\quad\delta=\max_{u\ne a,b}x_u,\quad U=V(G)\setminus\{a,b\},\quad Z=\sum_{u\in U}x_u^{3},
\]
Then $A,B\to\nu$, $\delta\to0$ and $Z\to2/3$. Let $r_a=|U\setminus N(a)|$, $r_b=|U\setminus N(b)|$ and $r=r_a+r_b$.

\begin{claim}\label{cl:3common}
$r=o(N)$.
\end{claim}

\begin{proof}
Put $C=N(a)\cap N(b)\cap U$ and $c=|C|$. By~\eqref{eq:facesp3}, faces avoiding $a,b$ have total weight at most $200^{1/3}\delta n^{1/3}=o(N^{1/3})$. The at most two faces containing both hubs have weight $o(1)$. For the other faces, the link-cycle estimate in Claim~\ref{cl:3master} and H\"older's inequality give a bound $\sum_{u\in U}c_ux_u^{2}\le(\sum_{u\in U}c_u^{3})^{1/3}Z^{2/3}$, where $c_u=A\mathbf1_{au\in E(G)}+B\mathbf1_{bu\in E(G)}$. Moreover, $\sum_{u\in U}c_u^3\le c(A+B)^3+(N-c)\max\{A,B\}^3$. If $c/N\to\theta<1$ along a subsequence, then
\[
\limsup\frac{\lambda}{3N^{1/3}}\le\nu(1+7\theta)^{1/3}(2/3)^{2/3}<\frac{2^{4/3}}3,
\]
contrary to Claim~\ref{cl:3stab}. Thus $N-c=o(N)$, and $r\le2(N-c)$.
\end{proof}

\begin{claim}\label{cl:3edge}
$ab\in E(G)$ for all large $n$.
\end{claim}

\begin{proof}
Suppose $ab\notin E(G)$, and put $C=N(a)\cap N(b)$, so $|C|=N-o(N)$. Let $E_a$ be the edges $cd$ of $G[C]$ for which $acd$ is a face, and define $E_b$ similarly. If $C=N(a)$, then $E_a$ consists of the entire link cycle of $a$, so $|E_a|=|C|$. Otherwise, $C$ forms $t_a$ maximal blocks on that cycle. Every neighbor of $a$ outside $C$ is missed by $b$, so $t_a\le|N(a)\setminus C|\le r_b$, and $|E_a|=|C|-t_a\ge|C|-r_b$. The same argument gives $|E_b|\ge|C|-r_a$. Planarity of $G[\{a,b\}\cup C]$ gives $e(G[C])\le|C|$. Therefore $|E_a\cap E_b|\ge|C|-r>0$. For $cd\in E_a\cap E_b$, the faces $acd$, $bcd$ form a quadrilateral $acbd$ with empty interior. Replacing $cd$ by $ab$ gives a triangulation with faces $abc$, $abd$ in place of $acd$, $bcd$. At $x$, the change in the face polynomial is
\[
3(A+B)(x_c+x_d)\Bigl(\frac{AB}{A+B}-\frac{x_cx_d}{x_c+x_d}\Bigr)>0,
\]
since $x_cx_d/(x_c+x_d)\le\delta\to0$, whereas $AB/(A+B)\to\nu/2$. This contradicts extremality.
\end{proof}

The edge $ab$ lies on faces $abp$, $abq$. Removing $b$ from the link cycle of $a$ leaves a path $P_a$ from $p$ to $q$ on $N(a)\cap U$, whose edges correspond to the faces containing $a$ but not $b$; define $P_b$ similarly. Put $\tau=AB/(A+B)$ and $F(P)=\sum_{uv\in E(P)}x_ux_v+\tau(x_p+x_q)$. The faces meeting $\{a,b\}$ have total weight $AF(P_a)+BF(P_b)$, since $(A+B)\tau=AB$ accounts for $abp$, $abq$. The path $P_a$ has $d_a=N-1-r_a$ edges and mass $Z-M_a$, where $M_a=\sum_{u\in U\setminus N(a)}x_u^3$; define $d_b$, $M_b$ similarly and put $M=M_a+M_b$. Thus $F(P_a)\le\Theta_{d_a,\tau}(Z-M_a)$ in the notation of Lemma~\ref{lem:fan}. If $S_0$ is the weight of faces avoiding both hubs, then
\[
\frac\lambda3=AF(P_a)+BF(P_b)+S_0 .
\]

\begin{claim}\label{cl:3compare}
$S_0\ge R:=A\bigl(\Theta_{N-1,\tau}(Z)-F(P_a)\bigr)+B\bigl(\Theta_{N-1,\tau}(Z)-F(P_b)\bigr)$.
\end{claim}

\begin{proof}
Assign $A,B$ to the hubs of $D_N$ and a maximizer of $\Theta_{N-1,\tau}(Z)$ to its path. The vector has mass one and value $3(A+B)\Theta_{N-1,\tau}(Z)$. Extremality and the preceding identity give the claim.
\end{proof}

We will show that $R>S_0$ whenever $r>0$. The relevant scales are $r/\lambda^2$ and $\lambda M$. The comparison uses $8\gamma/3>4/9$ and $1/12<1/6$, where $\gamma$ is from Lemma~\ref{lem:fan}$(iv)$. Indeed, $2^{1/3}>5/4$ gives $\gamma>81/400>1/6$.

\begin{claim}\label{cl:3Msmall}
$M=o(1)$.
\end{claim}

\begin{proof}
There are $2N-r$ faces meeting $a$ or $b$, so exactly $r$ faces avoid both. By~\eqref{eq:facesp3}, $S_0\le(4r)^{1/3}=o(N^{1/3})$. For $h\in\{a,b\}$, use $F(P_h)\le\Theta_{d_h,\tau}(Z-M_h)$ and $d_h\le N-1$. Lemma~\ref{lem:fan}$(iii)$ gives
\[
\Theta_{N-1,\tau}(Z)-F(P_h)\ge N^{1/3}\bigl[Z^{2/3}-(Z-M_h)^{2/3}\bigr]-K
\]
for a constant $K$. If, after interchanging the hubs, $M_a\ge c_0>0$ along a subsequence, concavity of $s\mapsto s^{2/3}$ gives the lower bound $\frac23c_0N^{1/3}-K$ for the $a$ term, while the $b$ term is at least $-K$. Since $A,B\to\nu$, eventually $R\ge\frac{\nu c_0}3N^{1/3}-2K$, contradicting $S_0\ge R$.
\end{proof}

Suppose, for a contradiction, that $r>0$ along a subsequence; then $M>0$. Put $C=N(a)\cap N(b)\cap U$ and $E=U\setminus C$. For $u\in E$, let $\mu(u)\in\{1,2\}$ count the hubs not adjacent to $u$. Then $r=\sum_{u\in E}\mu(u)$ and $M=\sum_{u\in E}\mu(u)x_u^{3}$. Let $\FA^{(j)}$ be the faces avoiding both hubs with exactly $j$ vertices in $C$. We have $\FA^{(3)}=\emptyset$, since three pairwise adjacent vertices of $C$, together with $a,b$, would form $K_5$. Each vertex $u\in E$ has at most two neighbors in $C$, since three would form $K_{3,3}$ with $\{a,b,u\}$. Consequently, $u$ lies on at most four faces of $\FA^{(1)}\cup\FA^{(2)}$. Call $f=cdu\in\FA^{(2)}$, with $c,d\in C$, \emph{good} if $\sigma_f:=x_cx_d\le\nu^2/\lambda$, and let $\BB$ be the other faces of $\FA^{(2)}$.

For $cd\in E(G[C])$ and $h\in\{a,b\}$, let $D_h(cd)$ be the closed disc bounded by the triangle $hcd$ whose interior avoids the other hub; see Figure~\ref{fig:disc}.

\begin{figure}[H]
\centering
\begin{tikzpicture}[scale=0.9,
hub/.style={circle,fill=black,inner sep=2.4pt},
pv/.style={circle,draw=black,fill=white,inner sep=1.5pt}]
\fill[black!8] (0,2.2) -- (-1.9,0) -- (1.9,0) -- cycle;
\fill[black!8] (0,-2.2) -- (-1.9,0) -- (1.9,0) -- cycle;
\fill[black!22] (-1.9,0) -- (0,1.05) -- (1.9,0) -- cycle;
\fill[black!22] (-1.9,0) -- (0,-1.05) -- (1.9,0) -- cycle;
\node[hub] (a) at (0,2.2) {}; \node[above] at (a) {$a$};
\node[hub] (b) at (0,-2.2) {}; \node[below] at (b) {$b$};
\node[pv] (c) at (-1.9,0) {}; \node[left] at (c) {$c$};
\node[pv] (d) at (1.9,0) {}; \node[right] at (d) {$d$};
\node[pv] (u) at (0,1.05) {}; \node[right] at (u) {$u$};
\node[pv] (v) at (0,-1.05) {}; \node[right] at (v) {$v$};
\draw (a)--(c)--(b)--(d)--(a); \draw (c)--(d);
\draw (a) to[out=190,in=170,looseness=2.8] (b);
\draw (a)--(u); \draw (c)--(u)--(d);
\draw (b)--(v); \draw (c)--(v)--(d);
\node at (-0.42,1.30) {$D_f$};
\node at (-0.42,-1.30) {$D_g$};
\node at (0,0.40) {$f$}; \node at (0,-0.40) {$g$};
\end{tikzpicture}
\caption{Faces $f=cdu$ and $g=cdv$ of $\FA^{(2)}$ inside the discs $D_f=D_a(cd)$ and $D_g=D_b(cd)$.}\label{fig:disc}
\end{figure}

\begin{claim}\label{cl:3disc}
For $f=cdu\in\FA^{(2)}$, exactly one disc $D_f=D_h(cd)$ contains the interior of $f$. Every vertex in the interior of $D_f$ lies in $E$. The interiors of the discs associated with distinct faces of $\FA^{(2)}$ are disjoint, and $P_h\cap D_f$ is a segment $c=w_0,w_1,\dots,w_m=d$ of $P_h$ with $m\ge2$. These segments have pairwise disjoint interiors.
\end{claim}

\begin{proof}
The edges $ab,ac,ad,bc,bd,cd$ form an embedded $K_4$. Its two faces incident with $cd$ have closures $D_a(cd)$, $D_b(cd)$, and $f$ lies on the inner side of $cd$ in exactly one of them. If an interior vertex of $D_f$ were adjacent to the other hub $h'$, that edge would cross the triangle $hcd$. Thus all interior vertices lie in $E$, and $D_f$ is the closure of a face of $S=G[\{a,b\}\cup C]$. Its boundary determines $h$ and $cd$, and only one face of $G$ lies immediately on the inner side of $cd$. Distinct faces of $\FA^{(2)}$ therefore give distinct faces of $S$, whose interiors are disjoint. The edges from $h$ into $D_f$ are consecutive between $hc$ and $hd$ in the rotation at $h$. Their ends form a segment of its link cycle avoiding $h'$, hence a segment of $P_h$. If $m=1$, the face between $hc$ and $hd$ would fill $D_f$, leaving no room for $u$.
\end{proof}

For $f\in\BB$, with disc $D$ and segment $w_0,\dots,w_m$, put $Q_D=\sum_{i=0}^{m-1}x_{w_i}x_{w_{i+1}}$ and $T_D=\sum_{v\in\operatorname{int}(D)}x_v^3$. Split $\BB$ into $\BB_<$ and $\BB_\ge$ according as $Q_D<\sigma_f/2$ or $Q_D\ge\sigma_f/2$. For each $h\in\{a,b\}$, replace simultaneously the segment associated with every face of $\BB_<$ whose disc has hub $h$ by the edge $cd$. By Claim~\ref{cl:3disc}, the segment interiors are disjoint and lie in $E$, whereas their ends lie in $C$. The result $\widehat P_h$ is a path with the same ends. Put $g_h=F(\widehat P_h)-F(P_h)$, and let $q_h$, $T_h$ be the number and mass of the deleted vertices. Then $q_h\le r$ and $T_h\le M$.

\begin{claim}\label{cl:3S0}
With $\eta_n=2\delta/\min\{A,B\}\to0$,
\[
S_0\le\Bigl(\frac1{12}+o(1)\Bigr)\lambda M+\Bigl(\frac49+o(1)\Bigr)\frac r{\lambda^{2}}+\eta_n(Ag_a+Bg_b),\qquad g_a,g_b\ge0 .
\]
\end{claim}

\begin{proof}
All vertices of a face of $\FA^{(0)}$ lie in $E$, so~\eqref{eq:facesp3} bounds the weight of this family by $(4r)^{1/3}M=o(\lambda M)$. For nonnegative $\xi,\eta,\zeta,\sigma,t$,
\[
\xi\eta\zeta\le\frac\lambda{48}(\eta^3+\zeta^3)+\frac{256}{3\lambda^2}\xi^3,\qquad\sigma t\le\frac\lambda{48}t^3+\frac8{3\sqrt\lambda}\sigma^{3/2},
\]
These inequalities follow by maximizing $\frac\xi2t^2-\frac\lambda{48}t^3$ and $\sigma t-\frac\lambda{48}t^3$ over $t\ge0$. Apply the first to $cuv\in\FA^{(1)}$, where $c\in C$. The terms other than $\frac\lambda{48}(x_u^3+x_v^3)$ total at most $256\delta^3r/(3\lambda^2)=o(r/\lambda^2)$. Apply the second to a good face, with $\sigma=\sigma_f$ and $t=x_u$. Its remaining term is at most $8(\nu^2/\lambda)^{3/2}/(3\sqrt\lambda)=4/(9\lambda^2)$, and there are at most $r$ such faces. Each kept term $\frac\lambda{48}x_u^3$ occurs at most four times for $u\in E$, so their sum is at most $\lambda M/12$.

For $f\in\BB_<$, replacing the segment by $cd$ gains $g_D=\sigma_f-Q_D>\sigma_f/2$, so $\prod_{v\in f}x_v\le\delta\sigma_f\le2\delta g_D$. The gains add, giving $g_h\ge0$ and a total contribution at most $\eta_n(Ag_a+Bg_b)$. For $f\in\BB_\ge$, write $Q_D=Q_\partial+Q_{\rm int}$, where $Q_\partial=x_cx_{w_1}+x_{w_{m-1}}x_d\le(x_c+x_d)T_D^{1/3}$ and $Q_{\rm int}\le m^{1/3}T_D^{2/3}$. At least one of them is at least $\sigma_f/4$. Using $x_u\le T_D^{1/3}$, $\lambda\sigma_f>\nu^2$ and $m\le r+1$, we obtain
\[
\prod_{v\in f}x_v\le\lambda T_D\max\Bigl\{\frac{64\delta^2}{\nu^2},\frac{4(r+1)^{1/3}}{\lambda}\Bigr\}=o(\lambda T_D).
\]
The disc interiors are disjoint subsets of $E$, so $\sum_DT_D\le M$.
\end{proof}

\begin{claim}\label{cl:3R}
$\displaystyle R\ge\Bigl(\frac{8\gamma}3-o(1)\Bigr)\frac r{\lambda^2}+\Bigl(\frac16-o(1)\Bigr)\lambda M+Ag_a+Bg_b$.
\end{claim}

\begin{proof}
Fix $h\in\{a,b\}$. The path $\widehat P_h$ has $N-1-e_h$ edges, where $e_h=r_h+q_h\le2r=o(N)$, and mass $W_h=Z-m_h$, where $m_h=M_h+T_h\le2M=o(1)$. Thus $F(P_h)+g_h=F(\widehat P_h)\le\Theta_{N-1-e_h,\tau}(W_h)$, and
\[
\Theta_{N-1,\tau}(Z)-F(P_h)\ge\bigl[\Theta_{N-1,\tau}(Z)-\Theta_{N-1,\tau}(W_h)\bigr]+\bigl[\Theta_{N-1,\tau}(W_h)-\Theta_{N-1-e_h,\tau}(W_h)\bigr]+g_h .
\]
Here $\tau\to\nu/2$, $Z,W_h\to2/3$, and $m_h\le Z/2$ for large $n$, so the uniform estimates in Lemma~\ref{lem:fan} apply. Part $(v)$ bounds the first bracket below by $(\frac23-o(1))N^{1/3}Z^{-1/3}m_h$. Telescoping part $(iv)$ over $e_h$ steps, whose indices lie between $N/2$ and $N$, bounds the second below by $(\gamma-o(1))Z^{2/3}e_hN^{-2/3}$. Multiply by $A$ for $h=a$ and by $B$ for $h=b$, add, and discard the nonnegative terms involving $q_h,T_h$. Now $Ar_a+Br_b=(\nu+o(1))r$, $AM_a+BM_b=(\nu+o(1))M$, and
\[
\lambda^2\nu Z^{2/3}N^{-2/3}\to\frac83,\qquad\frac{\frac23\nu N^{1/3}Z^{-1/3}}{\lambda}\to\frac16.
\]
\end{proof}

\begin{proof}[Proof of Theorem~\ref{thm:main}$(b)$]
Put $\kappa=8\gamma/3-4/9>0$. For large $n$, Claims~\ref{cl:3S0} and~\ref{cl:3R}, with $\eta_n\le1/2$, give
\[
R-S_0\ge\frac\kappa2\cdot\frac r{\lambda^2}+\frac1{24}\lambda M+\frac12(Ag_a+Bg_b)>0
\]
whenever $r>0$, contradicting Claim~\ref{cl:3compare}. Thus, for large $n$, both hubs are adjacent to every vertex of $U$, and $ab\in E(G)$ by Claim~\ref{cl:3edge}. It follows that $e(G[U])=3N-1-2N=N-1$. The $N-1$ edges of $P_a$, which spans $U$, lie in $G[U]$, so $G[U]=P_a$ and $G=K_2+P_N$. Its $2N$ faces are the hyperedges of $D_N$. Since an extremal hypergraph exists, $D_{n-2}$ is the unique one up to isomorphism.
\end{proof}

The restriction to large $n$ cannot be removed. Form $G_*$ from $K_4$ by inserting a vertex into each face and joining it to the three boundary vertices, so $n=8$. By symmetry and uniqueness of the normalized Perron vector, the original and inserted vertices have coordinates $\alpha,\beta$, respectively. The eigenequations $\lambda\alpha^2=6\alpha\beta$ and $\lambda\beta^2=3\alpha^2$ give $\lambda(\HH(G_*))^3=108$. On $D_6$, take the positive vector $z$ with coordinate $113$ at each hub and coordinates $84,92,94,94,92,84$ along the path. Direct integer evaluation gives
\[
119z_v^2-25\sum_{e\ni v}\prod_{u\in e\setminus\{v\}}z_u\ge584>0\qquad\text{for every }v.
\]
For the Perron vector $x$ of $D_6$, choose $v$ maximizing $x_v/z_v$ and apply~\eqref{eq:eigen}. This yields $\lambda(D_6)<119/25<108^{1/3}$, so $D_6$ is not extremal for $n=8$.

\section{Concluding Remarks}\label{sec:remark}

For $k=4$, Theorem~\ref{thm:main}$(c)$ holds for every $n\ge5$. The other parts require sufficiently large $n$; the arguments allow thresholds to be extracted, but we have not computed them. For ordinary graphs, both conjectures are settled with explicit finite lists of exceptions: the fan is extremal among outerplanar graphs for every $n\ne6$~\cite{LN}, and $K_2+P_{n-2}$ is extremal among planar graphs for every $n\ge9$~\cite{LNW}. Hypergraphs also have small-order exceptions, as $D_6$ shows at the end of Subsection~\ref{sec:p3}. Whether the fan $k$-angulation is extremal for every admissible $n$, with an explicit finite list of exceptions, remains open for every $k\ge3$.

The graph problem has also been studied on other surfaces. Hong~\cite{H1,H2} bounded the spectral radius in terms of genus, and Ellingham and Zha~\cite{EZ} sharpened the bound to $2+\sqrt{2n+8\gamma-6}$ for Euler genus $\gamma$. Zhai, Fang and Lin~\cite{ZFL} proved that, for fixed $\gamma$ and large $n$, every extremal graph is obtained from $K_2+P_{n-2}$ by adding $3\gamma$ edges inside a core of bounded order. The two common-incidence bounds used here rest on the nonplanarity of $K_{3,3}$. On a surface of Euler genus $\gamma$, analogous bounds can depend on $\gamma$, since $K_{3,2\gamma+3}$ does not embed there. The graph results suggest that genus may change a bounded configuration around the dominant vertices without changing their number. Identifying the extremal configuration for $k$-angulations of a surface of Euler genus $\gamma\ge1$ remains open for every $k\ge3$.

\section*{Declaration of generative AI and AI-assisted technologies in the manuscript preparation process}

During the preparation of this work, the authors used Claude (Anthropic) for language refinement, technical editing, and computational verification of examples. The authors reviewed and edited the output as needed and take full responsibility for the content of the published article.


\begin{thebibliography}{99}

\bibitem{BR}
B.N. Boots and G.F. Royle, A conjecture on the maximum value of the principal eigenvalue of a planar graph, {\it Geograph. Anal.} {\bf 23} (1991) 276--282.

\bibitem{BH}
A.E. Brouwer and W.H. Haemers, {\it Spectra of Graphs}, Universitext, Springer, New York, 2012.

\bibitem{CV}
D. Cao and A. Vince, The spectral radius of a planar graph, {\it Linear Algebra Appl.} {\bf 187} (1993) 251--257.

\bibitem{CPZ}
K.C. Chang, K. Pearson and T. Zhang, Perron--Frobenius theorem for nonnegative tensors, {\it Commun. Math. Sci.} {\bf 6} (2008) 507--520.

\bibitem{CD}
J. Cooper and A. Dutle, Spectra of uniform hypergraphs, {\it Linear Algebra Appl.} {\bf 436} (2012) 3268--3292.

\bibitem{CR}
D. Cvetkovi\'c and P. Rowlinson, The largest eigenvalue of a graph: a survey, {\it Linear Multilinear Algebra} {\bf 28} (1990) 3--33.

\bibitem{ELW}
M.N. Ellingham, L. Lu and Z. Wang, Maximum spectral radius of outerplanar 3-uniform hypergraphs, {\it J. Graph Theory} {\bf 100} (2022) 671--685.

\bibitem{EZ}
M.N. Ellingham and X. Zha, The spectral radius of graphs on surfaces, {\it J. Combin. Theory Ser. B} {\bf 78} (2000) 45--56.

\bibitem{FGH}
S. Friedland, S. Gaubert and L. Han, Perron--Frobenius theorem for nonnegative multilinear forms and extensions, {\it Linear Algebra Appl.} {\bf 438} (2013) 738--749.

\bibitem{GR}
C. Godsil and G. Royle, {\it Algebraic Graph Theory}, Graduate Texts in Mathematics {\bf 207}, Springer, New York, 2001.

\bibitem{GH}
B.D. Guiduli and T.P. Hayes, On the spectral radius of graphs of given genus, preprint, 1998.

\bibitem{H0}
Y. Hong, A bound on the spectral radius of graphs, {\it Linear Algebra Appl.} {\bf 108} (1988) 135--139.

\bibitem{H1}
Y. Hong, On the spectral radius and the genus of graphs, {\it J. Combin. Theory Ser. B} {\bf 65} (1995) 262--268.

\bibitem{H2}
Y. Hong, Upper bounds of the spectral radius of graphs in terms of genus, {\it J. Combin. Theory Ser. B} {\bf 74} (1998) 153--159.

\bibitem{Li}
L.-H. Lim, Singular values and eigenvalues of tensors: a variational approach, in: {\it Proceedings of the IEEE International Workshop on Computational Advances in Multi-Sensor Adaptive Processing}, vol.\ 1, Puerto Vallarta, 2005, pp.\ 129--132.

\bibitem{LN}
H. Lin and B. Ning, A complete solution to the Cvetkovi\'c--Rowlinson conjecture, {\it J. Graph Theory} {\bf 97} (2021) 441--450.

\bibitem{LNW}
L. Liu, B. Ning and Y. Wang, A complete solution to the Boots--Royle/Cao--Vince conjecture, preprint, arXiv:2607.19268, 2026.

\bibitem{Q}
L. Qi, Eigenvalues of a real supersymmetric tensor, {\it J. Symbolic Comput.} {\bf 40} (2005) 1302--1324.

\bibitem{SW}
A.J. Schwenk and R.J. Wilson, On the eigenvalues of a graph, in: L.W. Beineke and R.J. Wilson (eds.), {\it Selected Topics in Graph Theory}, Academic Press, London, 1978, pp.\ 307--336.

\bibitem{TT}
M. Tait and J. Tobin, Three conjectures in extremal spectral graph theory, {\it J. Combin. Theory Ser. B} {\bf 126} (2017) 137--161.

\bibitem{W}
D.B. West, {\it Introduction to Graph Theory}, 2nd ed., Prentice Hall, Upper Saddle River, NJ, 2001.

\bibitem{ZFL}
M. Zhai, L. Fang and H. Lin, Spectral extremal graphs on closed surfaces of fixed Euler genus, preprint, arXiv:2601.16380, 2026.

\bibitem{Z}
A.A. Zykov, Hypergraphs, {\it Uspekhi Mat. Nauk} {\bf 29} (1974) 89--154 (in Russian).

\end{thebibliography}
\end{document}